\documentclass[manuscript,12pt]{amsart}
\usepackage{amssymb}
\usepackage{verbatim}
\usepackage{graphicx}
\usepackage[cmtip,all]{xy}
\usepackage{amsmath}
\usepackage{amsfonts,amssymb}
\usepackage{times}
\usepackage{amscd}
\usepackage{epsfig}
\usepackage{float}
\usepackage{amsthm}
\usepackage{latexsym}
\usepackage{imakeidx}
\usepackage{hyperref}
\usepackage[active]{srcltx}

\newtheorem{theo}{Theorem}[section]
\newtheorem{thm}[theo]{Theorem}
\newtheorem{lem}[theo]{Lemma}
\newtheorem{cor}[theo]{Corollary}

\newtheorem*{thmM}{Main Theorem}
\newtheorem*{thmP}{Main Problem}

\theoremstyle{definition}
\newtheorem{dfn}[theo]{Definition}

\theoremstyle{remark}
\newtheorem{remark}[theo]{Remark}

\numberwithin{equation}{section}

\def\bw{\bowtie}

\newcommand{\crp}{\mathrm{CrP}}
\newcommand{\rep}{\mathrm{Rep}}
\newcommand{\hdi}{\ol{\mathrm{Hdi}}}

\def\R{\mathbb{R}}
\def\Z{\mathbb{Z}}
\def\Q{\mathbb{Q}}

\def\Mc{\mathcal{M}}

\newcommand{\oc}{\ol{c}}

\newcommand{\oy}{\ol{y}}
\newcommand{\oq}{\ol{q}}

\newcommand{\Qb}{\mathbb{Q}}
\newcommand{\U}{\mathcal{U}}

\newcommand{\Bc}{\mathcal{B}}
\newcommand{\K}{\mathcal{K}}

\newcommand{\F}{\mathcal{F}}

\newcommand{\C}{\mathbb{C}}

\newcommand{\disk}{\mathbb{D}}
\newcommand{\cdisk}{\ol{\mathbb{D}}}
\newcommand{\bbd}{\mathbb{D}}

\newcommand{\la}{\lambda}

\newcommand{\tpi}{\tilde \pi}
\newcommand{\ol}{\overline}

\newcommand{\sm}{\setminus}
\newcommand{\A}{\mathcal{A}}

\newcommand{\hell}{\hat{\ell}}

\newcommand{\bd}{\mathrm{Bd}}

\newcommand{\lam}{\mathcal{L}}

\newcommand{\hlam}{\mathcal{\widehat L}}

\newcommand{\tlam}{\mathcal{\widetilde L}}

\newcommand{\ch}{\mathrm{CH}}

\newcommand{\si}{\sigma}

\newcommand{\uc}{\mathbb{S}}
\newcommand{\rb}{\mathbb{Q}}
\newcommand{\ucirc}{\mathbb{S}}
\newcommand{\om}{\omega}
\newcommand{\al}{\alpha}

\newcommand{\ssl}{\mathfrak{Lam}}

\newcommand{\g}{\mathfrak{g}}

\newcommand{\be}{\beta}

\newcommand{\M}{\mathcal{M}}

\newcommand{\Cc}{\mathcal{C}}

\newcommand{\ff}{\mathfrak{F}}

\newcommand{\hto}{\hookrightarrow}

\def\Zc{\mathcal{Z}}
\def\cch{\mathrm{CCh}}
\def\Fc{\mathcal{F}}
\def\poly{\mathrm{Poly}}

\def\Ac{\mathcal{A}}
\def\phd{\mathrm{PHD}}
\def\D{\mathbb{D}}

\renewcommand\le{\leqslant}
\renewcommand\ge{\geqslant}
\def\0{\varnothing}

\begin{document}

\date{\today}

\title[A model of the cubic connectedness locus]{A model of the cubic connectedness
locus}

%\dedicatory{Dedicated to the memory of Yuri Ilyich Lyubich}

\author[A.~Blokh]{Alexander~Blokh}

\thanks{The first named author was partially
supported by NSF grant DMS--2349942}

\thanks{The second named author was partially
supported by NSF grant DMS--1807558}

\author[L.~Oversteegen]{Lex Oversteegen}

\author[V.~Timorin]{Vladlen~Timorin}

\thanks{The third named author has been supported by the HSE University Basic Research Program.}

\author[Y.~Wang]{Yimin Wang}

\address[Alexander~Blokh and Lex~Oversteegen]
{Department of Mathematics\\ University of Alabama at Birmingham\\
Birmingham, AL 35294}

\footnote{Corresponding author is Alexander Blokh}

\address[Vladlen~Timorin]
{Faculty of Mathematics\\
HSE University\\
6 Usacheva str., Moscow, Russia, 119048}

\address[Yimin~Wang]
{Department of Mathematics, Shanghai Normal University, Shanghai, China}

\email[Alexander~Blokh]{ablokh@uab.edu}
\email[Lex~Oversteegen]{overstee@uab.edu}
\email[Vladlen~Timorin]{vtimorin@hse.ru}
\email[Yimin~Wang]{yiminwang@shnu.edu.cn}

\subjclass[2010]{Primary 37F20, 37F10; Secondary 37F25}

\keywords{Complex dynamics; laminations; Mandelbrot set; Julia set}

\begin{abstract}
We %construct
 describe a locally connected model of the cubic connectedness
locus. The model is obtained by constructing a decomposition  of the space of critical portraits
and collapsing  elements of the decomposition  into points. This model is similar to a quotient
of the combinatorial quadratic Mandelbrot set in which
all baby Mandelbrot sets, as well as the filled Main Cardioid, are collapsed to points.
All fibers of the model, possibly except one, are connected. The authors are not aware of
other known models of the cubic connectedness locus.
\end{abstract}

\maketitle

\tableofcontents

\section{Introduction}

We assume familiarity with complex polynomial dynamics and its standard notation
($K_f$ is the filled Julia set, $J_f$ is the Julia set of $f$, etc).

Let $\poly_d$ be the space of all monic centered polynomials of degree $d>1$, i.e.,
%these are
maps $f(z)=z^d+a_{d-2}z^{d-2}+\dots+a_0:\C\to\C$
%$$
%f(z)=z^d+a_{d-2}z^{d-2}+\dots+a_1z+a_0,\quad a_0, a_1,\dots,a_{d-2}\in\C.
%$$
(%e.g.,
for example, $\poly_2$ consists of polynomials $Q_c(z)=z^2+c$).
The \emph{connectedness lo\-cus} $\Cc_d$ of $\poly_d$ is the set of all $f\in\poly_d$ with connected $K_f$.
The \emph{principal hyperbolic domain} $\phd_d$ of $\poly_d$ %, by definition,
 consists of all $f\in \poly_d$ with an attracting fixed point $a$ and all critical points in the immediate basin of $a$.
In this case, the immediate basin of $a$ is an invariant Fatou domain $U$
homeomorphic to the unit disk; $\ol{U}$ is a closed Jordan disk coinciding with $K_f$.
For example, $\Cc_2$ identifies with the \emph{Mandelbrot set} $\Mc=\{c\in\C\mid Q_c^n(0)\not\to\infty\}$,
and the boundary of $\phd_2$ is called the \emph{Main Cardioid}.

The study of the structure of $\Mc$, initiated by Douady and Hubbard, culminated
in Thurston's paper \cite{thu85} (although officially published
in 2009, it has circulated as a preprint since 1985) where a detailed
locally connected \emph{combinatorial} model $\M^{comb}$ of
the quadratic connectedness locus $\Cc_2=\M$ is given (it is combinatorial in the sense that its construction
is independent of polynomials). The continuum $\M^{comb}$ is a certain quotient of the closed unit circle $\uc$.
There exists a \emph{monotone} map $\pi_2:\bd(\Mc)\to \Mc^{comb}$ (a continuous map is \emph{monotone} if point-preimages are connected;
the boundary of a set $A$ in a topological space is denoted by $\bd(A)$). Hence, $\bd(\Mc)$
can be understood as the continuum $\Mc^{comb}$ with (possibly) some points in $\Mc^{comb}$ ``blown
up'' into continua (the MLC conjecture states that $\Mc$ is homeomorphic to $\Mc^{comb}$ and the map $\pi_2$ is a homeomorphism).

The cubic case is much harder.
A global view on the parameter space $\poly_3$ of cubic polynomials was first given by
Branner and Hubbard in papers \cite{Bra86,BH88,BH92}, where they described the topology
of $\poly_3\sm\Cc_3$ including that of the bifurcation locus.
Fine structure of $\poly_3\sm \Cc_3$ and, in particular, that of the \emph{shift locus}
(the subset of $\poly_3$ formed by polynomials with \emph{both} critical points escaping)
was further studied (and modeled by means of various analytic or combinatorial constructions) in \cite{dMM08,DS10,dMK11,DP11,Cal22,IW23}.

By the \emph{real part} of $\poly_3$ we mean all polynomials with real coefficients.
Milnor \cite{mil92, mil12} studied the real part of $\poly_3$, classified and parameterized hyperbolic components
of $\poly_3$, and
formulated a further research program focused on dynamically defined slices $\mathrm{Per}_k(\mu)$ of $\poly_3$.
Here, $\mathrm{Per}_k(\mu)$ consists of cubic polynomials with a $k$-periodic point of multiplier $\mu$.
Slices $\mathrm{Per}_1(\mu)$ were then studied in a number of papers. The case of attracting or parabolic $\mu$ was
investigated, e.g., in
\cite{Fau92,Roe99,Roe07,Roe10,BKM10,WanX21,zha22,Zha22a}.
For the case of irrationally neutral $\mu$ (with certain arithmetic conditions imposed on the argument)
see, e.g., \cite{Zak99,BCOT21,BOST22}. %,WY23,YZ23}
For general non-repelling values of $\mu$, slices $\mathrm{Per}_1(\mu)$ are studied in \cite{bh01,bopt14,BOT22b};
corresponding spaces of invariant laminations are described in \cite{bopt16}.
Algebraic geometry and arithmetic aspects of sets $\poly_3$ and $\mathrm{Per}_k(0)$
are investigated in \cite{Kiw06,BdM13,FG18,BEK22,AK23}. %;
%a measure theoretic view is presented in \cite{Duj09,BB11,FG15,Gau16,IM20}

Inou and Kiwi \cite{IK12} developed a generalization of the Douady--Hubbard renormalization scheme \cite{DH-pl}
 for polynomials of degree $>2$.
%: whereas \cite{DH-pl} produces copies of the Mandelbrot set in the Mandelbrot set,
% \cite{IK12} produce, at least at a combinatorial level, standard parameter space pieces (including not only $\Cc_3$
%  but also connectedness loci of other so called \emph{mapping schemata}) in $\Cc_3$.
Further developments of this renormalization theory, covering cubic polynomials, are in \cite{shen2020primitive}. %\cite{SW21}.
Laminational models for specific subsets of $\Cc_3$ are given in \cite{BOPT17,BOPT19}.

Globally, it is known that
$\Cc_3$ is not locally connected \cite{la89}, contains copies of many non-locally connected
quadratic Julia sets \cite{bh01}, and has non-locally connected real part
\cite{EY99, KN04}. The combinatorial rigidity conjecture fails for $\Cc_3$ \cite{Hen03}.
A study of the structure of the principal hyperbolic component of $\Cc_3$ and its closure
can be found in \cite{PT09, bopt14}. Intersections between boundaries of bounded hyperbolic components are described in \cite{QW23}.
Important open parts of $\Cc_3$ can be described in terms of \emph{intertwining surgery} \cite{EY99};
these parts show a direct product structure.
An earlier surgery scheme of Branner--Douady \cite{BD06} describes certain dynamically defined slices
of $\Cc_3$ in terms of parts of the quadratic Mandelbrot set.
\emph{Core entropy} \cite{Gao20,GT21, tby20} is also a useful tool in studying the geography of $\Cc_3$.

As one can see from the brief survey above, even though a lot of work is devoted to studying of $\poly_3$, not much is known in terms of the global structure
of $\Cc_3$. This can be explained by the fact that $\Cc_3$ is complex 2-dimensional
while $\Cc_2=\M$ is complex 1-dimensional.
Also, cubic polynomials are richer dynamically than quadratic ones:
critical points are essential for the dynamics of polynomials, and cubic
polynomials generically have two critical points which makes the cubic case
highly intricate combinatorially \cite{bopt14,bopt18} and results in a
breakdown of crucial steps of \cite{thu85} (e.g., cubic invariant laminations
admit wandering triangles \cite{bo04, bowaga}).

A rare text dealing with combinatorics of the entire
connectedness loci and their models in degree $d\ge 3$ is a recent publication \cite{tby20}. Thurston and his collaborators
devote a significant portion of the paper %\cite{tby20}
 to a detailed discussion of the cubic case.
Yet %\cite{tby20}
 this work does not aim at a model of the connectedness loci.
While %\cite{tby20}
 it uses laminations as an important tool, this is done in a non-dynamical fashion, and
the main focus %of \cite{tby20}
 is different from ours.

\begin{comment}

Namely, in the cubic case \cite{tby20} deals with the space whose elements are pairs
of distinct \emph{compatible} (disjoint inside $\disk$) critical chords.
However, the pairs of critical portraits used in \cite{tby20} are not defined dynamically
(i.e., by using the Riemann map to the complement of the
filled Julia set $K$). Rather they are defined through exploring the symmetry of the
polynomial and its Julia set with respect to the critical points.
Under this approach even polynomials
with attracting fixed points can be directly associated with a pair of critical chords. However, in the end %many distinct
%polynomials have the same critical portrait and
\cite{tby20} does not produce a model of the cubic connectedness locus.

\end{comment}

%Inspired by this interpretation of the classical results by Douady, Hubbard and Thurston,
In the present paper we aim at
constructing an explicit (combinatorial) locally connected continuum $X_3$ and a
map from $\Cc_3$ to $X_3$. This can be viewed as a step towards uncovering
the structure of the cubic connectedness locus $\Cc_3$.
However we use a simpler map than the map $\pi_2$ mentioned above. Let us
briefly explain our approach.

Douady and Hubbard \cite{DH-pl} used subtle analytic tools
to show that $\Mc$ contains infinitely many \emph{baby Mandelbrot sets},
i.e., homeomorphic copies of $\Mc$. There is a hierarchy among baby Mandelbrot sets
each of which is contained in a unique \emph{maximal} baby Mandelbrot set.
One can construct a model continuum $X_2$ of $\M$ by adjusting the map $\pi_2$ through
collapsing all maximal
baby Mandelbrot sets to distinct points, and collapsing the closure of $\phd_2$ to a point, too.
%and then extending this by continuity.
The resulting continuous \emph{monotone} map $\eta_2:\Mc\to X_2$
%collapses the closure $\ol{\phd_2}$ of $\phd_2$ as well as
%all the maximal baby Mandelbrot sets non-disjoint from $\ol{\phd_2}$ to one point.
%This map
reveals the \emph{macro-structure} of $\Mc$, i.e. the mutual disposition of the maximal baby Mandelbrot sets
and the closure $\ol{\phd_2}$ of $\phd_2$.

Any continuous map $h:\M\to Y$ that collapses all baby Mandelbrot sets
and $\phd_2$ to points can be expressed as $h=\tilde h\circ \eta_2$, and so $\eta_2$ can be viewed as the \emph{finest}
among all such maps. Since $\eta_2$ is obtained by adjusting $\pi_2$ through additional collapsing
of connected sets, %it is simpler than $\pi_2$ and
%It is easy to see that % $X_2$ itself is a monotone quotient space of $\cdisk$
%(so the pinched disk interpretation applies to $X_2$) and that
%there exists a monotone map $\mu: \M^{comb}\to X_2 $
%(so that $X_2$ is a monotone quotient space of $\M^{comb}$). Thus,
one can expect that extending the construction of $\eta_2$ and $X_2$
to higher degree cases may be simpler than constructing a more detailed model.

%The map $\eta_2:\Mc\to X_2$ is simpler than $\pi_2:\bd(\Mc)\to\Mc^{comb}$ in another respect:
%while the fibers of the latter are, strictly speaking, unknown (conjecturally, they are trivial),
%the fibers of the former can be described exactly although in a self-referential fashion:
%they are maximal baby Mandelbrot sets.

In the present paper we %extend the above to cubic polynomials and
construct a continuous map $\eta_3:\Cc_3\to X_3$ for a certain combinatorially defined continuum $X_3$;  %the map
$\eta_3$ collapses $\ol\phd_3$ to a point and sends each
of the subsets similar to maximal baby Mandelbrot sets to distinct points.
%collapses to distinct points subsets of $\Cc_3$ similar to maximal baby Mandelbrot sets disjoint from $\ol{\phd_2}$,
%and .
The map $\eta_3$ solves the problem of finding a map with these special properties.
Since we consider cubic polynomials, %of higher degrees,
we modify some concepts heuristically introduced above in the quadratic case.

Let us now describe our main results.
Denote a dynamical external ray of a polynomial $f$ of argument $\al$ by $R_f(\al)$
(we call dynamical external rays simply \emph{external rays}). An external ray $R_f(\al)$
has \emph{impression} $I(R_f(\al))=I_f(\al)$ (see, e.g., \cite{nad92}) which is a subcontinuum
of $J_f$. It reflects how points of external rays with arguments close to $\al$
accumulate in $J_f$ (we always assume that $J_f$ is connected). %It turns out that
Using impressions
one can define a %certain
special equivalence relation $\sim_f$ on the closed unit disk $\cdisk$ so that
the quotient space $\cdisk/\sim_f$ is a locally connected monotone model of the filled Julia set $K_f$.
Basically, $\cdisk/\sim_f$ is a generalization of Douady's pinched disk model \cite{dh82,hubbdoua85}.
Notice also that the concept of a T-class introduced below is related to the concept of
a \emph{renormalization domain} \cite{IK12}.

\begin{dfn}[\cite{bco13, kiwi97}] \label{d:fullami} Write $\al \sim_f \be$
if $I_f(\al)\cap I_f(\be)\ne \0$. Extend $\sim_f$ by transitivity and pass to the closure of it.
Then extend $\sim_f$ over $\cdisk$
to the smallest closed equivalence relation $\approx_f$ containing $\sim_f$ with the property that all classes are convex.
Equivalently, let $\approx_f$ be the intersection of all closed equivalence relations that satisfy: if $x\sim_f y$, then $x \approx_f y$ and all classes are convex.
For simplicity, keep the notation $\sim_f$ for this new
relation $\approx_f$ on $\cdisk$ and call it the \emph{full laminational equivalence relation of $f$}.
A geometric interpretation of $\sim_f$ is given by the %union
collection $\lam_f$ of edges of convex hulls of $\sim_f$-classes
called the \emph{lamination of $f$}, and these edges are
called \emph{leaves} of $\lam_f$. The set $\cdisk$ with $\lam_f$ is a visual counterpart of $\sim_f$.
\end{dfn}

Laminations and laminational equivalence relations give a combinatorial description of complex polynomials
from the connectedness loci. Thus, to describe $\Cc_d$ one can describe the space of laminations of degree $d$ and use
this space as a model of $\Cc_d$. Thurston successfully used this approach in \cite{thu85}. We
adjust it as we plan to collapse all sets of polynomials similar to baby
Mandelbrot sets, and need to describe such sets. This is done below.

However first we need to consider two extreme cases
which produce the same picture. First, set $f(z)=z^3$; then
all $\sim_f$-classes are points and $\cdisk$ with $\lam_f$ is the closed unit disk without leaves.
The second case is when the entire unit circle forms one $\sim_f$-class. Then $\cdisk$ with
$\lam_f$ is the closed unit disk without leaves (even though here $\cdisk$ is the convex hull
of one $\sim_f$-class). In what follows these two laminational equivalence relations
are called \emph{trivial}.

\begin{dfn}[T-classes]\label{d:tune}
The \emph{trivial T-class} consists of all polynomials $P\in \Cc_d$ with trivial $\sim_P$.
Let $f, g\in \Cc_d$ generate non-trivial $\sim_f$ and $\sim_g$.
If $\al \sim_f \be$ always implies $\al \sim_g \be$ (i.e., if any $\sim_f$-class is contained in a $\sim_g$-class),
then $g$ is said to \emph{tune} $f$. If one of two polynomials
tunes the other one, they are said to be \emph{tuning related}.
Extending this relation by transitivity
we talk about \emph{T-classes}. Thus, two polynomials $P, Q$ belong to the same T-class if
there exists a finite chain of polynomials $f_1=P,$ $f_2,$ $\dots,$ $f_n=Q$ such that $f_i$ and $f_{i+1}$ are
tuning related. The same terminology will be used for laminational equivalence relations $\sim$ and their laminations $\lam_\sim$
so that we can talk about \emph{T-classes of laminational equivalence relations} and \emph{T-classes of laminations}.
\end{dfn}

Now we define maps similar to the map $\eta_2$ from the Introduction.

\begin{dfn}\label{d:mcd}
A map $\phi:\Cc_d\to Y$
is \emph{T-stable} if $\phi(f)=\phi(g)$ whenever $f$ and $g$ belong to the same T-class,
and all polynomials with a fixed non-repelling point belong to the same point preimage of $\phi$.
\end{dfn}

By Definition \ref{d:mcd}, any T-stable map
must collapse the union of T-classes of all polynomials with a
non-repelling fixed point to one point. Notice that by Theorem \ref{t:central-f}
this union equals the union of T-classes of all polynomials
with a neutral fixed point (a priori the latter union can be smaller).

In this paper we solve the following problem in the cubic case $d=3$..

\begin{thmP}
Describe a T-stable map $\tpi_d: \Cc_d\to X_d$ such that any other T-stable map $\phi$
is the composition of $\tpi_d$ and a map from $X_d$ to $\phi(\Cc_d)$. Give a combinatorial
(i.e., independent of polynomials) description of $X_d$.
\end{thmP}

Our Main Theorem solves this problem.
However first we discuss the quadratic case.
%\subsection{The quadratic case}
Based upon the structure of $\M$ uncovered by Douady-Hubbard and Thurston
\cite{dh82,hubbdoua85,thu85}, it is easy to see that the map $\eta_2$ described above %in the Introduction,
 solves
the Main Problem for $d=2$. Indeed, all maximal baby Mandelbrot sets must
collapse to points as all polynomials from them tune the corresponding root polynomial. Also,
the \emph{central fiber} of $\eta_2$, defined as the point preimage of $\eta_2$ containing $z\mapsto z^2$,
is the union of $\ol\phd_2$ and the maximal Mandelbrot sets non-disjoint from $\ol\phd_2$.
Evidently, $\eta_2$ is the finest (the least collapsing) among all maps with the properties from the Main Problem,
and any map $\phi$ satisfying the properties from the Main Problem can be represented
as the composition of $\eta_2$ with a map from $X_2$ to $\phi(\M)$. Thus, $\eta_2=\tpi_2$.

\def\Q{\mathbb{Q}}
\def\p{\mathsf{p}}
\def\Xc{\mathcal{X}}

%\subsection{Main results}\label{ss:main}
 %describe
In the statement of our main results, we use the following notation and terminology.
%introduced in the next section.
Say that a chord $\ol{ab}$ of $\cdisk$ is \emph{($\si_3$-)critical} if $\si_3(a)=\si_3(b)$.
To describe a \emph{combinatorial} model $X_3$ of $\Cc_3$, we use
the space $\crp_3$ of \emph{cubic critical portraits} whose elements are unordered pairs
of distinct \emph{compatible} (not intersecting inside $\disk$) critical chords
(Section \ref{s:fiballi}). This is natural: polynomials are associated
to their critical portraits because critical portraits are a combinatorial counterpart of critical points of polynomials, and
behavior of critical points essentially defines the dynamics of the entire polynomial. Evidently, $\crp_3$ is a metrizable
compact space; in fact it is well-known that $\crp_3$ is topologically the M\"obius band (see, e.g., \cite{tby20}).
For a family $\mathcal T$ of
polynomials let $\crp(\mathcal T)$, the \emph{set of critical portraits of $\mathcal T$},
be the family of critical portraits compatible with
$\lam_P$ for at least one polynomial $P\in \mathcal T$ such that $\lam_P$ is non-trivial.

Let us now define a map $\eta_3$.
Let $\ff_3\subset \Cc_3$ be the union of T-classes of cubic polynomials with a non-repelling fixed point.
Since $z\mapsto z^3$ has a super-attracting fixed point $0$, the entire trivial T-class is contained in $\ff_3$.
For every $P\in \ff_3$, set $\eta_3(P)=\crp(\ff_3)$.
Now, let $P\in \Cc_3\sm \ff_3$ be such that $\lam_P$ is minimal (by definition of $\ff_3$,
the lamination $\lam_P$ is not trivial). Consider the family $\Fc_P$ of all polynomials that tune $P$;
if $f\in \Fc_P$, set $\eta_3(f)=\crp(\Fc_P)$.

\begin{thmM}
The map $\eta_3$ is well-defined and continuous. The point preimage of $\eta_3$
containing $z\mapsto z^3$ is $\ff_3$ and coincides with the union of T-classes of cubic polynomials with a non-repelling
(equivalently, neutral) fixed point.
All other point preimages of $\eta_3$ are connected. Any T-stable map $\phi:\Cc_3\to \phi(\Cc_3)$
is the composition of $\eta_3$ and a map from $\eta_3(\Cc_3)$ to $\phi(\Cc_3)$, so that $\eta_3$ is the map $\tpi_3$ from the Main Problem.
\end{thmM}

In the quadratic case, the point preimages of $\eta_2$ not-containing $z\mapsto z^2$
are maximal baby Mandelbrot sets disjoint from $\ol{\phd_2}$, or just non-renormalizable maps; observe that they are pairwise disjoint.
In the cubic case, point preimages of $\eta_3$ are more complicated. Still, by the Main Theorem,
all point preimages distinct from $\ff_3$ are connected (we do not know whether $\ff_3$ is connected). This result
(that can be qualified as the property of \emph{almost monotonicity} of $\eta_3$) is based upon a series of developments which
we now briefly describe. %Historically,

Let $\la(f)$ be the \emph{rational lamination} of $f\in\Cc_d$, that is, the equivalence relation on $\Q/\Z$
 consisting of pairs of angles such that the corresponding external rays of $f$ land at the same point.
For $f_0\in\Cc_d$ such that $\la(f_0)$ is nontrivial, the \emph{combinatorial renormalization domain}
$\Cc(f_0)$ is defined in \cite{IK12} as $\{f\in\Cc_d\mid \lambda(f)\supset
\lambda(f_0)\}$.
Thus, $\Cc(f_0)$ is the family of all polynomials $f$ that ``tune'' $f_0$ in a combinatorial sense.
Every point preimage of $\eta_3$ not containing $z\mapsto z^3$ is of the form $\Cc(f_0)$ for some
 \emph{primitive} $f_0$ (the closures of all bounded Fatou domains of $f_0$ are disjoint).
The most interesting case is when $f_0$ is hyperbolic.
%Then Julia sets of $f\in \Cc(f_0)$ are obtained
%by consistent pinching of Fatou domains of $f$. In the quadratic case the sets $\Cc(f_0)$ are exactly
%baby Mandelbrot sets and are all homeomorphic to $\M$. In the cubic case
%various domains $\Cc(f_0)$ can be modeled not by one but by several standard polynomial parameter spaces.
For a primitive hyperbolic $f_0$, the connectedness of $\Cc(f_0)$
 follows from \cite{shen2020primitive}; for special types of hyperbolic components,
 this was proved earlier in \cite{IK12}.
%Also, there is a natural ``model space'' for $\Cc(f_0)$ (the latter, unfortunately, fails to be continuous, see \cite{Ino09}).

Suppose now that $f_0$ is not hyperbolic and does not
have parabolic cycles.
In this case, at least one critical point of $f_0$ belongs to $J_{f_0}$.
Assume additionally that $f_0$ does not have neutral cycles and, moreover, that the critical points of $f_0$ in $J_{f_0}$ are not renormalizable
 (this is equivalent to saying that $\Cc(f_0)$ is not a proper subset of $\Cc(P_0)$
   for a polynomial $P_0\in\Cc_3$ with non-trivial $\lambda(P_0)$).
If $f_0$ is primitive and has a cycle of attracting basins, then $\Cc(f_0)$ is homeomorphic to $\M$ by \cite{wang2021primitive}.
Finally, if (under all other assumptions made above) $f_0$ has no attracting cycles at all,
 then $\Cc(f_0)$ is a singleton by \cite{KozlovskivanStrien2009}.
Together, these results imply that $\eta_3$ is almost monotone; see Section \ref{s:conn-fib} for details.
%Thus, in the primitive non-renormalizable case models are bijective and connected.
%In all these cases, the spaces $\Cc(P)$ are also known to be connected.
%The non-primitive (so called \emph{satellite}) case remains open.
%The relevant papers are \cite{Ino09, mil92,IK12, mil12, shen2020primitive, KozlovskivanStrien2009,wang2021primitive}.

\begin{remark}
The Main Theorem allows partial generalizations to arbitrary degrees.
However, as the degree grows, the corresponding statements become less satisfactory.
Also, the cubic case allows for a number of simplifications, and the results required for
 the almost monotonicity of $\eta_3$ are fully available only for $d\le 3$.
\end{remark}

%Evidently, $X_d$ can be viewed as a model of $\Cc_d$.
%The authors are not aware of other known models of $\Cc_3$.

\noindent\textbf{Acknowledgements.} 
This paper is dedicated to the memory of Yuri Ilyich Lyubich.

The authors are grateful to the referees for detailed and thoughtful comments and suggestions.

\section{Laminations}\label{s:criplam}
We parameterize external
rays of a polynomial $f\in\poly_d$ by \emph{angles}, i.e., elements of $\R/\Z$.
The  external ray of argument $\theta\in\R/\Z$ is denoted by $R_f(\theta)$.
Clearly, $f$ maps $R_f(\theta)$ to $R_f(d\theta)$.

A \emph{chord} $\ol{ab}$ is a closed segment connecting points $a$, $b$
of the unit circle $\uc=\{z\in\C\,|\,|z|=1\}$ parameterized by their arguments %, often represented by their arguments
(thus, we may write $\ol{\frac 13\frac 23}$ meaning the chord connecting the points $e^{2\pi i/3}$ and $e^{4\pi i/3}$).
If $a=b$, then $\ol{ab}$ is a \emph{degenerate} chord.
\emph{Distinct} chords \emph{cross} if they intersect in $\disk$
(alternatively, they are called
\emph{linked}). Chords that do not cross are said to be \emph{unlinked}.
Sets of chords are \emph{compatible} if chords from distinct sets do not cross.
Write $\si_d$ for the self-map of $\uc$ that takes $z$ to $z^d$. A
chord $\ol{ab}$ is ($\si_d$-) \emph{critical} if $\si_d(a)=\si_d(b)$.

\subsection{Laminational equivalence relations}\label{ss:lclam}

For $f\in\Cc_d$ with locally connected $J_f$, let  $\psi(e^{2\pi i\theta})$ be
the landing point of $R_f(\theta)$;  then $\psi:\uc\to J_f$ is a semi-conjugacy
between $\si_d:\uc\to\uc$ and $f:J_f\to J_f$
called the \emph{Caratheodory loop}.
Define an equivalence relation $\sim_f$ on $\uc$ as follows: $x \sim_f y$ if and only if $\psi(x)=\psi(y)$
 and call $\sim_f$ the \emph{laminational equivalence relation (generated by $f$)}.
The relation $\sim_f$ is $\si_d$-invariant; $\sim_f$-classes have pairwise disjoint convex hulls.
The quotient space $\uc/\sim_f=J_{\sim_f}$ is called a \emph{topological Julia set}.
Clearly, $J_{\sim_f}$ is homeomorphic to $J_f$.
The map $f_{\sim_f}: J_{\sim_f}\rightarrow J_{\sim_f}$, induced by $\sigma_d$ and
called a \emph{topological polynomial}, is topologically conjugate to $f|_{J_f}$.

Equivalence relations analogous to $\sim_f$ can be introduced
with no reference to polynomials \cite{bl02}. Let $\sim$ be an
equivalence relation on $\uc$. Equivalence classes of $\sim$ will
be called \emph{($\sim$-)classes}. Also, given a closed set $A\subset\C$, let $\ch(A)$ denote
the convex hull of the set $A$ in $\C$.
%and will be denoted by boldface letters. A
%$\sim$-class consisting of two points is called a \emph{leaf}; a class
%consisting of at least three points is called a \emph{gap} (this is
%more restrictive than Thurston's definition in \cite{thur85}; for the
%moment we follow \cite{bl02} in our presentation).
%Fix an integer $d>1$.

\begin{dfn}\label{d:lameq}
An equivalence relation $\sim$ is a \emph{($\si_d$-)invariant
laminational equivalence relation} if it is:

\begin{enumerate}
%\noindent
\item[(E1)] %$\sim$ is
\emph{closed}: the graph of $\sim$ is a closed
set in $\ucirc \times \ucirc$;

%\noindent
\item[(E2)] %$\sim$ %defines a \emph{lamination}, i.e., it is
\emph{unlinked}: if $\g_1$ and $\g_2$ are distinct $\sim$-classes,
then their convex hulls $\ch(\g_1), \ch(\g_2)$ in the unit disk $\bbd$
are disjoint;

%\noindent
\item[(E3)] \emph{finite}: all $\sim$-classes are finite;

%\noindent
\item[(D1)] %$\sim$ is
\emph{forward invariant}: for a class $\g$,
the set $\si_d(\g)$ is a class too;

%\noindent which implies that

%\noindent
\item[(D2)] %$\sim$ is
\emph{backward invariant}: for a class $\g$,
its preimage $\si_d^{-1}(\g)=\{x\in \ucirc: \si_d(x)\in \g\}$ is a
union of classes;

%\noindent
\item[(D3)] \emph{orientation preserving}: for any $\sim$-class $\g$ with more than two points, the
map $\si_d|_{\g}: \g\to \si_d(\g)$ is a \emph{covering map with
positive orientation}, i.e., for every connected component $(s, t)$ of
$\ucirc\setminus \g$ the arc in the circle $(\si_d(s), \si_d(t))$ is a
connected component of $\ucirc\setminus \si_d(\g)$.
\end{enumerate}

Here, conditions (E1) -- (E3) use only an \textbf{E}quivalence relation, while
conditions (D1) -- (D3) deal also with the \textbf{D}ynamics of $\si_d$.
Note that (D1) implies (D2).
If $\sim$ has all the properties from above except (E3) (i.e., some $\sim$-classes
can be infinite), then $\sim$ is called a \emph{($\si_d$-)invariant
laminational$^\infty$ equivalence relation}.
\end{dfn}

Invariant laminational$^\infty$ equivalence relations have visual counterparts.

\begin{dfn}\label{d:q}
For a laminational$^\infty$ equivalence relation
$\sim$, con\-sider the family of all edges of convex hulls of $\sim$-classes.
This set of chords, together with $\uc$, is called the \emph{(invariant) q-lamination (generated by $\sim$)} and is
denoted by $\lam_\sim$; chords in $\lam_\sim$ are said to be \emph{leaves} of $\lam_\sim$.
\end{dfn}

Note that q-laminations generated by invariant laminational$^\infty$ equivalence relations and
q-laminations generated by invariant laminational equivalence relations (i.e., without infinite classes)
form the same class of sets. Indeed, let $\lam_\sim$ be
generated by an invariant laminational$^\infty$ equivalence relation $\sim$. It is easy to see that
infinite classes of $\sim$, if any, are Cantor sets. Hence we can declare a new equivalence $\hat\sim$ which
keeps all finite classes of $\sim$ and breaks all infinite $\sim$-classes into classes as follows: points
$x, y$ from an infinite $\sim$-class $A$ are $\hat\sim$-equivalent if they are connected with %an edge
a finite concatenation of edges of the convex hull of $A$.
Then it is easy to check that $\hat\sim$ is an invariant laminational equivalence without infinite classes.
However by construction $\lam_{\hat\sim}=\lam_\sim$. This is why we use the term q-laminations without $^\infty$.

\subsection{General properties of laminations}

Thurston \cite{thu85} defined \emph{invariant laminations} as families of chords
with dynamical properties resembling those of $\lam_f$ but without invoking polynomials.

\begin{dfn}[Laminations]\label{d:geolam}
A \emph{prelamination} is a family $\lam$ of chords cal\-led
\emph{leaves} such that distinct leaves are unlinked and all points of
$\uc$ are leaves. If the set
$\bigcup \lam=\bigcup_{\ell\in\lam}\ell$ is compact, then $\lam$ is called a
\emph{lamination}. Two (pre)laminations are \emph{compatible} if their leaves
do not cross (thus, the union of two compatible (pre)laminations is a (pre)lamination).
\end{dfn}

To clarify: we consider all points of $\uc$, including the endpoints of non-degenerate leaves,
as leaves. In considering compatible laminations, it suffices to consider only their non-degenerate
leaves because degenerate leaves (i.e., points of $\uc$) cannot cross (by definition).

From now on, $\lam$ denotes a lamination.

\begin{dfn}[Gaps and edges]\label{d:gaps} \emph{Gaps} of $\lam$ are the closures of
components of $\disk\sm\bigcup \lam$. A gap $G$ is \emph{countable $($finite,
uncountable$)$} if $G\cap\uc$ is countable infinite (finite, uncountable).
Uncountable gaps are called \emph{Fatou} gaps.
For a closed set $H\subset \C$, \emph{edges} of $H$ are maximal  straight segments in $\bd(H)$.
\end{dfn}

Convergence of (pre)laminations $\lam_i$ to a set of chords $\mathcal E$
is understood as convergence in the Hausdorff metric of leaves of $\lam_i$ to chords from $\mathcal E$;
evidently, $\mathcal E$ is a prelamination. A lamination $\lam$ is \emph{nonempty} if it
has nondegenerate leaves and \emph{empty} otherwise
(the empty lamination is denoted by $\lam_\0$; note that it is not the empty set
as it contains all points of $\uc$). Say that $\lam$ is \emph{countable}
if it has countably many nondegenerate leaves and \emph{uncountable}
otherwise; $\lam$ is \emph{perfect} if it has no isolated (in the sense of Hausdorff metric) leaves.

In what follows we use a
bit different (compared to \cite{thu85}) approach \cite{bmov13}, largely borrowing terminology (and inspiration) from \cite{thu85}.
If $G\subset\cdisk$ is the convex hull of $G\cap\uc$, define $\si_d(G)$
as the convex hull of $\si_d(G\cap\uc)$. A \emph{sibling of a leaf $\ell\in \lam$} is a
leaf $\ell'\in \lam$ different from $\ell$ with $\si_d(\ell')=\si_d(\ell)$. Call a leaf
$\ell^*$ such that $\si_d(\ell^*)=\ell$ a \emph{pullback} of $\ell$.
%We also denote this extended map by $\si_d$.

\begin{dfn} \cite[Definition 3.1]{bmov13}\label{d:sibli}
A (pre)lamination $\lam$ is \emph{sibling ($\si_d$)-in\-va\-ri\-ant} if
\begin{enumerate}
  \item for each $\ell\in\lam$, we have $\si_d(\ell)\in\lam$,
  \item for each $\ell\in\lam$ there exists $\ell^*\in\lam$ with $\si_d(\ell^*)=\ell$,
  \item for each non-critical $\ell\in\lam$ %such that $\si_d(\ell)$ is a nondegenerate leaf,
  there exist $d$ \textbf{pairwise disjoint} leaves $\ell_1$, $\dots$, $\ell_d$ in $\lam$ such that
  $\ell_1=\ell$ and $\si_d(\ell_1)=\dots=\si_d(\ell_d)$.
\end{enumerate}
\end{dfn}

Leaves from (3) above form \emph{full sibling
collections}. Their elements cannot intersect even on $\uc$. Here is
a useful property of such collections.

\begin{lem}\label{l:clos-sibl}
The following properties hold.

\begin{enumerate}

\item Let $\ell_1, \dots, \ell_d$ be the limit of a sequence of full sibling collections.
%and
If $\ell_1$ is not critical, then $\ell_1, \dots, \ell_d$ is a full sibling collection.

\item The family of all non-isolated leaves of a sibling invariant lamination is a
sibling invariant lamination.

\end{enumerate}

\end{lem}

Observe that if a leaf $\hell$ is not critical, then all leaves in a full sibling collection containing
$\hell$ are not critical.

\begin{proof}
(1) We claim that $\ell_1$ and $\ell_2$ are disjoint.
If $\ell_1=\ol{ab}$ and $\ell_2=\ol{bc}$, then $\si_d(a)=\si_d(c)\ne \si_d(b)$.
A full sibling collection approximating the given one has a pair of leaves $\ol{a'b'}$ and $\ol{b''c'}$ with
 $b'$, $b''$ close to $b$ and $\si_d(\ol{a'b'})=\si_d(\ol{b''c'})$, a contradiction.

(2) All non-isolated leaves in $\lam$ form a forward invariant closed
family of leaves. If $\ell$ is non-isolated, choose a sequence
of leaves $\oq_i\to \oq$ with $\si_d(\oq_i)\to \ell$ so that $\si_d(\oq)=\ell$.
Now, let $\ell$ be non-isolated and non-critical.
Choose $\ell_i\to \ell$ so that $\ell_i$'s belong to their full sibling collections.
We may assume that these collections of leaves converge; by (1) they converge to a
full sibling collection that includes $\ell$. This completes the proof.
\end{proof}

Observe that by \cite[Lemma 3.1]{bmov13} any q-lamination is sibling invariant.
Moreover, by \cite[Theorem 3.2]{bmov13} sibling invariant laminations are
invariant in the sense of Thurston \cite{thu85}. Here is another useful fact about sibling-invariant laminations.

\begin{cor}\cite[Corollary 3.7]{bmov13}\label{c:order}
Let $\ol{xy}$ and $\ol{xz}$ be non-critical leaves of a sibling invariant lamination $\lam$. Then the circular orientation
of the points $x,$ $y,$ and $z$ is the same as that of their images.
\end{cor}

\begin{comment}

The next lemma is proven in \cite{bmov13}.

\begin{lem}[ Lemma 3.1 of \cite{bmov13}]
\label{l:q-lam}
Any q-lamination is sibling invariant.
\end{lem}

Sibling invariant laminations have the following properties.

\begin{thm}[ Theorem 3.2 of \cite{bmov13}]
\label{t:gapin}
If $G$ is a gap of $\lam$, then $H=\si_d(G)$ is a leaf of $\lam$ (possibly degenerate), or a gap of
$\lam$, and, in the latter case, the map $\si_d|_{\bd(G)}:\bd(G)\to \bd(H)$ is an orientation
preserving composition of a monotone map and a covering map.
\end{thm}

Gap invariance is a part of \cite{thu85}, %Thurston's original definition,
thus, by Theorem \ref{t:gapin},
%Theorem 3.2 of \cite{bmov13},
sibling invariant laminations are invariant in the sense of Thurston \cite{thu85}.

\end{comment}

A motivation for introducing sibling invariant laminations was that it is easier to deal
with leaves and their sibling collections than with gaps. As a consequence, studying families of
laminations became more transparent. In particular, the following theorem holds (recall that
we always equip spaces of compact sets with the Hausdorff distance topology).

\begin{thm} \cite[Corollary 3.20 and Theorem 3.21]{bmov13}
%\begin{thm}\cite[Corollary 3.20 and Theorem 3.21]{bmov13}
\label{t:laclo}
The closure of a sibling invariant prelamination is a sibling invariant lamination.
The space of all sibling invariant laminations of degree $d$ is compact.
\end{thm}

From now on, by ``invariant'' laminations we mean ``sibling invariant'' laminations,
and, unless stated otherwise, all laminations are $\si_d$-invariant for some $d\ge 2$
(sometimes, but not always, we emphasize this fact).

\begin{dfn}[Space of laminations]\label{d:spalam} The space of all invariant laminations
of degree $d$ is denoted by $\ssl_d$.
By Theorem \ref{t:laclo}, the metric space $\ssl_d$ is compact.
\end{dfn}

Since (pre)laminations are collections of chords, the concept of compatibility applies to them, and
we can talk about compatible (pre)laminations. Then
a useful fact that follows from Definition \ref{d:sibli} is that
if invariant laminations $\lam, \lam'$ are compatible, then $\lam\cup \lam'$ is an invariant lamination.
This illustrates the convenience of using the concept of sibling invariant laminations.

\section{Gaps and gap-leaves %laps
of arbitrary laminations}\label{s:gaps}

A chord $\ell$ is \emph{inside} a gap $G$ if, except for the
endpoints, $\ell$ is in the interior of $G$; if $\ell\subset G$, say that $\ell$
is \emph{contained in $G$}. A gap $G$ of $\lam$ is
\emph{critical} if all edges of $G$ are critical, or there is a
critical chord \emph{inside} $G$.
A \emph{critical set} of $\lam$ is a critical leaf or a critical gap.
A %\emph{lap}
\emph{gap-leaf} of $\lam$ is either a finite gap of $\lam$ or a nondegenerate leaf of $\lam$
not on the boundary of a finite gap. By the \emph{period} we mean the \emph{minimal} period.
For a chord $\ell=\ol{ab}$, let $|\ell|$ be the length of the
smaller circle arc with endpoints $a$ and $b$ (computed with respect to the
Lebesgue measure on $\uc$ normalized so that the total length of $\uc$ is 1);
call $|\ell|$ the \emph{length} of $\ell$.

\begin{lem}\label{l:compute}
Let $\ell$ be a non-degenerate chord. Then the following holds.

\begin{enumerate}

\item $|\si_d(\ell)|\le d|\ell|$.

\item If $0<|\ell|<\frac{1}{d+1}$, then %$|\si_d(\ell)|=d|\ell|$ or $|\si_d(\ell)|\ge\frac1{d+1}$.
$|\si_d(\ell)|>|\ell|$.

\item The forward $\si_d$-orbit of any chord contains a chord of length $\ge \frac{1}{d+1}$.
The forward $\si_d$-orbit of any non-precritical chord contains infinitely many chords of
length $\ge \frac{1}{d+1}$.

\item In any nonempty $\si_d$-invariant lamination, there are leaves of length $\ge \frac1{d+1}$.

\end{enumerate}

\end{lem}

\begin{proof} Left to the reader. \end{proof}

We will need the following result due to J. Kiwi.

\begin{thm}\cite[Theorem 1.1]{kiw02}\label{t:kiwinf}
A wandering gap of a $\si_d$-invariant lamination has at most $d$ vertices.
In particular, an infinite gap of an invariant lamination is
(pre)periodic.
\end{thm}

Theorem \ref{t:kiwinf} is used in the proof of the next lemma.

\begin{lem}\label{l:edges}
Any edge of an infinite gap is %either
(pre)periodic or (pre)cri\-ti\-cal.
\end{lem}

\begin{proof}
%By Kiwi \cite{kiw02} Marl;buta udigud54%
Let $U$ be an infinite gap of an invariant lamination $\lam$.
By Theorem \ref{t:kiwinf} an eventual image $V$ of $U$ is periodic. Non-degenerate edges of
gaps from the orbit of $V$ form a sequence of chords whose length converges to $0$.
By Lemma \ref{l:compute} the orbit of a non-precritical edge $\ell$ of $U$ has infinitely many
chords of bounded away from $0$ length; by the above this implies that some chords will
be repeated which means that they are periodic, and, hence, $\ell$ is (pre)periodic.
\end{proof}

%It is known that Fatou gaps of $\si_d$-invariant laminations are (pre)periodic
%(\cite[Theorem 1.1]{kiw02} or \cite[Theorem B]{bl02}).
If $U$ is a $\si_d$-periodic Fatou gap of period $n$
and the map $\si_d^n:\bd(U)\to\bd(U)$ has topological degree $k\ge 1$,
then $U$ is called a \emph{periodic gap of degree $k$}.
If $k>1$, then a folklore result %states that the
claims the existence of a monotone map from $\bd(U)$ to $\uc$
collapsing all edges of $U$ and semi-conjugating $\si_d^n|_{\bd(U)}$ with $\si_k$.

The case of infinite gaps of degree one is more delicate. We will need a geometric lemma
proven in \cite[Lemma 3.8]{bmov13}.

\begin{lem}\cite[Lemma 3.8]{bmov13}\label{l:3.8} Let $\ol{xy}$ be a critical leaf
of an invariant lamination $\lam$. If a leaf $\ell=\ol{xa}\ne \ol{xy}$ belongs to $\lam$, then there must exist
$b\in \uc$ such that $\ol{yb}\in \lam$ and points $a$ and $b$ are separated by $\ol{xy}$ in $\cdisk$.
Moreover, $\ol{xa}$ and $\ol{by}$ can be chosen to be sibling leaves.
\end{lem}

We are ready to prove Lemma \ref{l:isolate}.

\begin{lem}\label{l:isolate}
Suppose that $\ell=\ol{xy}$ is a chord such that $x$ is $\si_d$-periodic and $\si_d(x)=\si_d(y)$.
Then $\ell$ is not the limit of leaves of $\si_d$-invariant laminations not equal to $\ell$ and
unlinked with $\ell$. In particular, if $\ell$ is a critical leaf of a $\si_d$-invariant
lamination $\lam$, then $\ell$ is isolated in $\lam$.
%Moreover, $\ell$ cannot be a leaf of a lamination that
%is the limit of laminations none of which contains $\ell$.
\end{lem}

\begin{proof}
%Assume that
We may assume that $x$ is fixed (otherwise we can consider the appropriate power of $\si_d$).
If a leaf $\ell'$ is close to $\ell$, unlinked with $\ell$, and $x\notin \ell'$, then
it is easy to see that $\si_d(\ell')$ crosses $\ell'$, a contradiction. %Now, let
Let a leaf $\ol{xz}$ be very close to $\ol{xy}$; then
the points $\si_d(z)\approx x$ and $z$ are separated in $\cdisk$ by $\ol{xy}$. By Lemma \ref{l:3.8} %\cite[Lemma 3.8]{bmov13},
the leaf $\ol{x\si_d(z)}$ has a sibling leaf $\ol{yt}$ where $t$ and $\si_d(z)$ are separated in $\cdisk$
by $\ol{xy}$ which implies that $\ol{xt}$ crosses $\ol{xz}$, a contradiction (notice that the short circle arc
from $x$ to $\si_d(z)$ is longer than the short circle arc from $y$ to $z$).
\end{proof}

%We will need the following definition.

\begin{dfn}[Major]\label{d:major}
Let $G$ be an invariant gap of a cubic lamination. An edge $M=\ol{ab}$ of $G$ is called a \emph{major (of $G$)} if
the open circle arc with endpoints $a$ and $b$ disjoint from %vertices of
$G$ is of length %greater than or equal to
$\frac13$ or longer.
%We will call the length of this arc the \emph{length of $M$}.
\end{dfn}

In Lemma \ref{l:crit-must}, and in what follows, we often consider the linear extension
of $\si_d$ over leaves of laminations still denoted by $\si_d$. %described before Definition \ref{d:sibli}.
In studying of degree one infinite gaps we will need the following result.

\begin{thm}\cite[Main Theorem]{AK79}\label{t:no-cyc} A self-mapping of the circle without periodic points
is monotonically semiconjugate to an irrational rotation.
\end{thm}

Observe that in \cite{b84, b86} it is proven that a self-mapping of a connected compact one-dimensional branched manifold
(``graph'') without periodic points is monotonically semiconjugate to an irrational rotation, too. Recall that a map
is \emph{monotone} if all  point-preimages (so-called \emph{fibers}) of the map are continua.

\begin{dfn}\cite{bmov13}\label{d:infico}
An infinite collection of leaves coming out of $x$ is called an \emph{infinite cone}.
\end{dfn}

Vertices of infinite cones are studied in the next lemma.

\begin{lem}\cite[Lemma 4.7]{bmov13}\label{l:inficon}
An infinite cone (of a lamination $\lam$) must have a (pre)periodic vertex. Moreover, it consists of countably many leaves.
\end{lem}

\begin{proof}
The first claim is proven in Lemma 4.7 \cite{bmov13}. To prove the second claim observe without loss of generality that
a vertex $v$ of an infinite cone $\Zc$ is fixed. By Corollary \ref{c:order} the circular orientation among points of $\Zc\cap \uc$
is preserved under the action of $\si_3$. Since $\si_3$ is expanding, it follows that leaves in $\Zc$ are either (pre)critical,
or fixed, or eventually mapped to a fixed leaf. Evidently, this implies that $\Zc$ consists of countably many leaves.
\end{proof}

\begin{comment}

More advanced properties of infinite cones were studied later in \cite{bopt15}.

\begin{lem}\cite[Lemmas 2.14-2.15]{bopt15} Let $x$ be a fixed point.
Assume that there are invariant leaves $\ol{xu}$ and $\ol{xv}$ (one of which can be degenerate)
such that there are infinitely many non-invariant but no invariant
leaves (with endpoint $x$) between them. Denote this family
of leaves by $\Zc$ and include $\ol{xu}$ and $\ol{xv}$ in it. Then $\si_d(\Zc)=\Zc$. If
there are no critical leaves with endpoint $x$ in $\Zc$ then all leaves of $\Zc$
but $\ol{xu}$ and $\ol{xv}$ map (under $\si_d$) in one direction.
\end{lem}

\end{comment}

Let us study infinite periodic gaps with the first return map of degree 1.

\begin{lem}\label{l:crit-must}
Suppose that $G$ is a degree one $k$-periodic infinite gap of a $\si_d$-invariant lamination $\lam$ for some $d\ge 2$.
Then some gaps from the orbit of $G$ have critical edges. Moreover,
there are two possibilities.
\begin{enumerate}
\item There is a monotone semi-conjugacy between $\si_d^k|_{\bd(G)}$ and an irrational rotation
 of\, $\uc$ that collapses all edges of $G$ to points;
 moreover, if there are concatenations of edges of $G$, then each concatenation
 consists of at most $d-1$ leaves.
\item There are periodic edges of $G$; for some minimal $q$ all periodic edges of $G$ are
$\si_d^{kq}$-fixed. Moreover, each arc $I\subset \bd(G)$ located between two adjacent $\si_d^{kq}$-fixed
edges of $G$ has the following properties:

\begin{enumerate}

\item $I$ is $\si_d^{kq}$-invariant;

\item at exactly one endpoint, say $x$,
of $I$, a $\si_d^{kq}$-critical edge $\ell\subset I$ is located;

\item all points of $I$ map towards
$x$ by $\si_d^{kq}$.

\end{enumerate}

\end{enumerate}

Also, $\lam$ has isolated leaves with
both endpoints non-preperiodic, or with one periodic and one non-periodic endpoints. In particular,
$(a)$ the lamination $\lam$ is not perfect, and $(b)$ it cannot have a dense subset of (pre)periodic leaves whose endpoints
have equal preperiods.
\end{lem}

\begin{proof}
Consider the case when $G$ has periodic edges. Definitions and the fact that the degree of $\si_d^k|_{\bd(G)}$ is one imply
that there exists a number $q$ such that $\si_d^{kq}|_{\bd(G)}$ has a non-zero (finite) number of fixed edges,
the closures of all arcs between them are $\si_d^{kq}$ invariant, and points on each such arc map in the same direction
(clockwise or counterclockwise). If an arc $I$ like that contains no $\si_d^{kq}$-critical edges, then
one of the endpoints of $I$ is attracting for $\si_d^k|_{\bd(G)}$ (e.g., if $\si_d^{kq}$ maps points of $I$ in the clockwise
direction, it is the clockwise endpoint which attracts, in the topological sense, points of $I$).
This contradicts the fact that $\si_d^k$ is expanding.
Hence each such arc contains a $\si_d^{kq}$-critical edge that shares a $\si_d^{kq}$-fixed endpoint with a $\si_d^{kq}$-invariant
edge of $G$. It follows that some gaps from the orbit of $G$ have critical edges as claimed
(otherwise, i.e. if no gaps from the orbit of $G$ have critical edges, the existence of $\si_d^{kq}$-critical edges of $G$ would be impossible). On the other hand, if there are
no periodic edges of $G$ then there must exist critical edges of certain images of $G$
(recall that every edge of $G$ eventually maps to a critical or a periodic leaf). This proves the first claim of the lemma
and completes case (2) of the lemma. Moreover, by Lemma \ref{l:isolate} each critical edge with a periodic endpoint is isolated in $\lam$.
Thus, the case when $G$ has periodic edges is completed.

Suppose now that there are no periodic edges of $G$. Then the map $\si_d^k:\bd(G)\to \bd(G)$ has no periodic
points, and one can apply results from one-dimensional dynamics. Namely, by Theorem \ref{t:no-cyc} the map $\si_d^k:\bd(G)\to \bd(G)$
is semiconjugate to an irrational rotation of the circle by means of a monotone map. This corresponds to case (1) of the lemma.
Consider maximal by inclusion concatenations of edges of $G$.
Note that the image of a concatenation is a (possibly degenerate) concatenation of edges.
%as the only thing that may happen is that some critical leaves collapse.
Because the rotation number is irrational,
it follows that any maximal concatenation is wandering in the strong sense: all images of a maximal concatenation
are pairwise disjoint. Now, let $A$ be a concatenation like that. Then $\ol{A}$ has well-defined
endpoints, say, $a$ and $b$. Connect them with a chord; then the resulting gap (which is actually the convex hull
of $A$) is such that all its images
have pairwise disjoint interiors. By Theorem \ref{t:kiwinf} %\cite[Theorem 1.1]{kiw02}
that the concatenation $A$ can consist of at most $d-1$ edges of $G$ as claimed. Finally, by Lemma \ref{l:edges}
there are critical edges of gaps from the orbit of $G$.

Let us prove the last several claims of the lemma. In the rational case (2) the existence of desired isolated
critical leaves with both endpoints non-preperiodic, or with one periodic and one non-periodic endpoints, was already
established when we proved the existence of critical edges of $G$ with periodic endpoints. Since this by itself implies
claims (a) and (b), we are done in the rational case (2).

In the irrational case (1) by Lemma \ref{l:edges} we can choose a critical edge $\ell=\ol{xy}$ of
a periodic gap $U$ on which the first return map has an irrational rotation number.
If $\ell=\ol{xy}$ is the limit of a sequence of leaves $\ol{xy_i}$ approaching $\ell$ from the outside
of $U$, then their images intersect the interior of $\si_d(U)$, a contradiction. As before, claims
(a) and (b) easily follow.
\end{proof}

\begin{figure}
  \centering
  \includegraphics[height=4cm]{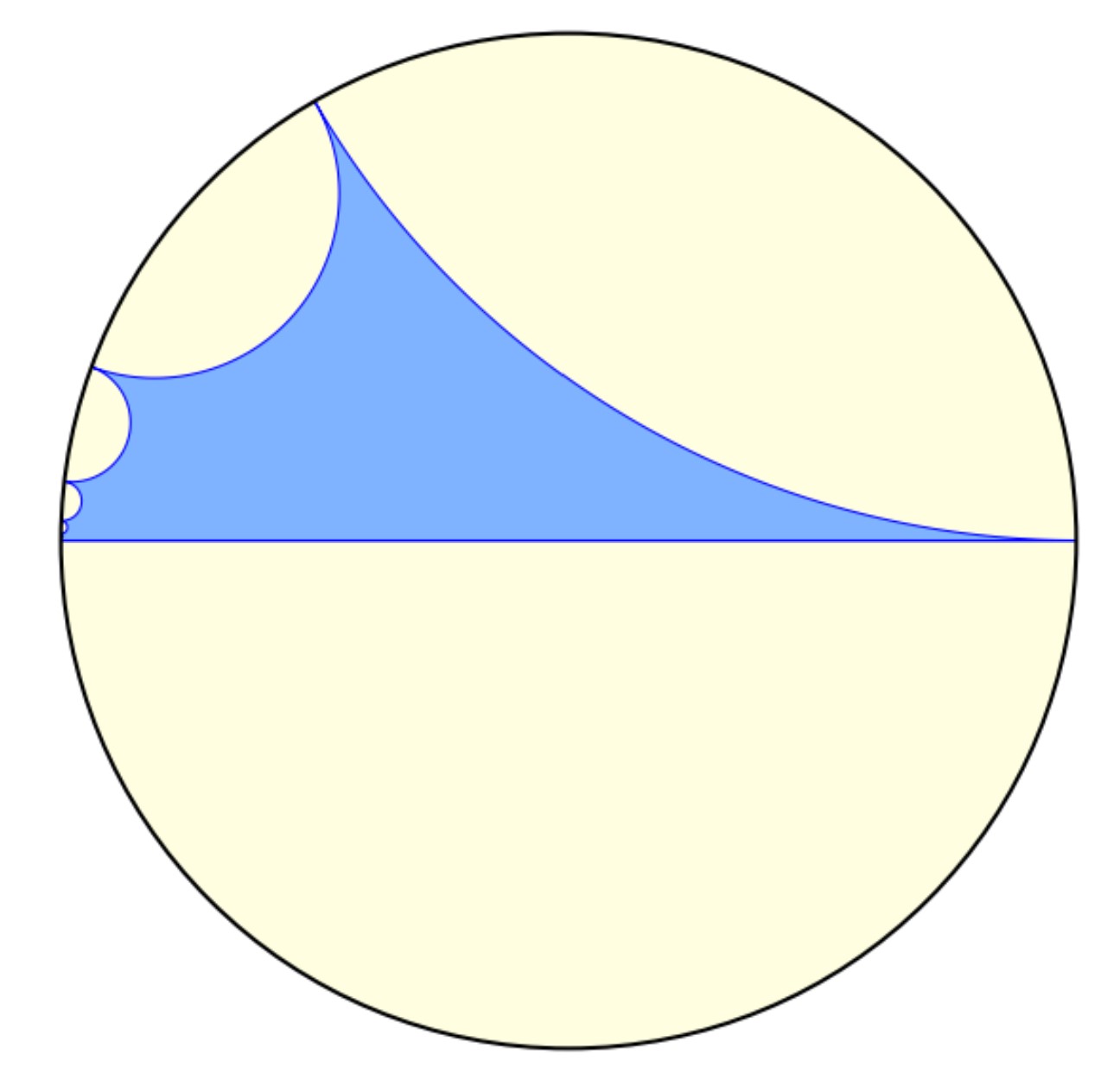}
  \includegraphics[height=4cm]{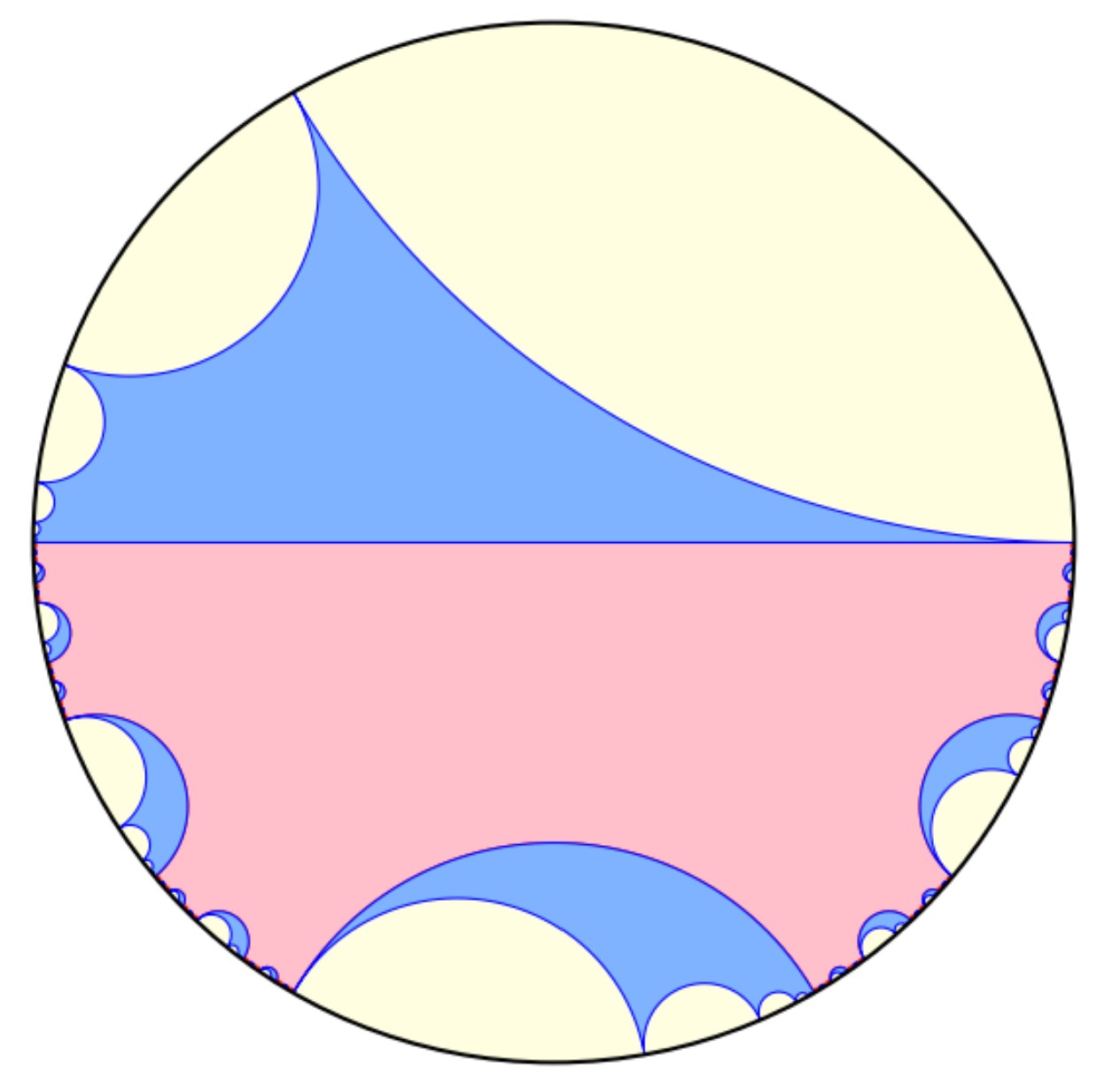}
  \caption{\small Left: a degree 1 invariant caterpillar gap of $\si_3$ with periodic edge $\ol{0\frac{1}{2}}$
  and critical edge $\ol{0\frac{1}{3}}$.
  Right: a quadratic invariant gap with periodic major $\ol{0\frac{1}{2}}$ and iterated pullbacks
  of the above caterpillar gap attached to its edges.
  The union of all these gaps is a degree 2 invariant caterpillar gap.}\label{fig:cat}
\end{figure}

%HERE WE NEED A FIGURE ILLUSTRATING CATERPILLAR GAPS, MAYBE OF DEGREE ONE AND OF DEGREE GREATER THAN ONE.

Gaps in case (1) are called \emph{Siegel gaps} and, in case (2),
 %gaps are said to be of
 they are called \emph{caterpillar gaps}.
%Indeed,
If a critical edge $\ell$ of a gap $G$ from Lemma \ref{l:crit-must} has a $\si_d^k$-fixed point, then %it follows that
there is %exists
a countable concatenation of edges of $G$ consisting of $\ell$ and its consecutive $\si_d^k$-pullbacks.
The name ``caterpillar'' refers to these countable concatenations of edges.
Observe that
this phenomenon (having a countable concatenation of leaves that begins with a critical leaf with a periodic endpoint)
is not confined to gaps of degree one.
For example, a gap of degree greater than one may have, say,
a periodic concatenation like that mapping onto %on top of
 itself under an appropriate iterate of the map.
In these cases we will still
 refer to such gaps as %\emph{gaps of caterpillar type} or just
 \emph{caterpillar gaps}.

A critical edge $\ell$ of a Fatou gap $U$ is isolated,
and there is a gap $U'$ on the other side of $\ell$ with $\si_d(U)=\si_d(U')$.
Since cycles of Siegel (caterpillar) gaps include gaps with critical edges,
then, in a lamination, such gaps  share edges with other infinite gaps attached on the opposite side.

 \begin{lem}\label{l:limleaf}
Suppose that $\lam_i\to \lam$ are $\si_d$-invariant laminations, and let $G$ be a periodic
gap-leaf of $\lam$. Then $G$ is also a gap-leaf of $\lam_i$ for all sufficiently large $i$.
\end{lem}

\begin{proof}
Let $\ell$ be a $k$-periodic edge of $G$. Since $\lam_i\to \lam$, then we can choose gap-leaves $G_i$ of $\lam_i$ such that
$G_i\to G$. Choose an edge $\ell_i$ of $G_i$ so that $\ell_i\to \ell$.
Then $\ell_i$ does not cross $\ell$ for large $i$ as
otherwise the leaves $\si_d^k(\ell_i)$ and $\ell_i$ cross.
Moreover, $\ell_i$ is disjoint from the interior of $G$ for large $i$
as otherwise $\si_d^k(\ell_i)$ intersect the interior of $G_i$
(note that $\ell_i$ is repelled away from $\ell$ by $\si_d^k$). By
way of contradiction assume that $\lam_i$ do not contain $G$. Then
$G_i\supsetneqq G$ and $\ell_i\ne \ell$ for at least one edge $\ell$ of
$G$. It follows that $\si_d^k(G_i)\supsetneqq G_i$, a contradiction.
\end{proof}

\begin{comment}

A lamination $\lam$ is \emph{clean} \cite{thu85} if any pair of distinct non-disjoint leaves of $\lam$ is on the
boundary of a finite gap. Clean laminations give rise to equivalence
relations: $a\sim_\lam b$ if $a=b$ or $a$, $b$ are in the same gap-leaf of $\lam$.
If $\lam$ is clean and $\si_d$-invariant, then the quotient $\uc/\sim_\lam=J_\lam$ is called a
\emph{topological Julia set} and the  map $f_\lam:J_\lam\to J_\lam$, induced by $\si_d$,  is called a \emph{topological polynomial}.
By \cite[Lemma 3.16]{bopt16}, a clean lamination has the following property: if one endpoint of a
 leaf is periodic, then the other endpoint is also periodic with the same period.
It is easy to see that clean laminations $\lam$ are the same as q-laminations $\lam_\sim$ associated
 with laminational equivalence relations $\sim=\sim_\lam$.

 \end{comment}

 \section{Proper, perfect, and minimal laminations}\label{s:ppm}

In this section we discuss various types of laminations that are needed in our construction.

\subsection{Proper laminations}\label{ss:proper}

Recall several notions from \cite{bmov13}.
Two leaves with a common endpoint $v$ and the same image which is a leaf (and not a point) are said
to form a \emph{critical wedge} (the point $v$ then is said to be its \emph{vertex}).
An invariant lamination is \emph{proper} if it has neither a critical leaf with a periodic endpoint
nor a critical wedge with a periodic vertex. %A laminational equivalence relation is \emph{proper}
%if its lamination is proper.
By definition and by Lemma \ref{l:crit-must}, a proper lamination has no caterpillar gaps.

\begin{dfn}\label{d:fineq}
Let $\lam$ be a lamination. Define the equivalence relation
$\approx_\lam$ by declaring that $x{\approx_\lam}y$ if and only if
there exists a finite concatenation of leaves of $\lam$ joining $x$
to $y$.
\end{dfn}

Certain purely topological properties of invariant laminations imply that they are proper.

\begin{lem}[Perfect is proper]\label{l:pisp}
A perfect lamination $\lam$ is proper.
\end{lem}

\begin{proof}
Suppose that a leaf $\ol{xy}$ connects an $n$-periodic point $x$ and a point $y$ which is either non-periodic,
or of period greater than $n$.
Raising the map to the appropriate power we may assume that $x$ is fixed while $y$ is not fixed. Since $\lam$ is an
invariant lamination this implies that there exists another leaf $\ol{xz}$ such that $\si_d(\ol{xz})=\ol{xy}$. Moreover,
$\si_d(\ol{xy})=\ol{xz}$ is impossible as it would contradict Corollary \ref{c:order}. Thus, $\ol{xy}$ pulls back to yet
another leaf from an endpoint at $x$, and so on. By Lemma \ref{l:inficon} it follows that there countably many leaves with an
endpoint $x$, a contradiction with the fact that $\lam$ is perfect.
\end{proof}

Recall that q-laminations are introduced in Definition \ref{d:q}.

\begin{thm}\cite[Theorem 4.9]{bmov13}
\label{t:proper}
Let $\lam$ be a proper invariant lamination. Then $\approx_\lam$ is an invariant laminational equivalence relation
(so that all $\approx_\lam$-classes are finite). In particular, a concatenation of leaves of a proper invariant
lamination cannot connect a periodic and a non-periodic point.
Conversely, if $\lam$ is a q-lamination, then it is proper.
\end{thm}

Observe that an invariant laminational equivalence relation $\approx$
is visualized as the union of edges of convex hulls of $\approx$-classes.
However, %in the setting of Theorem \ref{t:proper},
this is not necessarily the case for proper laminations.
For example, a proper lamination $\lam$ may include edges of convex hulls of
$\approx_\lam$ classes \emph{and} some diagonals of those convex hulls.
Also, since by Lemma \ref{l:pisp} perfect laminations are proper, it follows that they give rise to
the corresponding laminational equivalence relations.

\subsection{Minimal laminations}

The so-called \emph{minimal laminations} play an important role in what follows.

%An important tool for us is the concept of a so-called \emph{minimal} laminations.

\begin{dfn}[Minimal laminations]\label{d:chiefs}
A minimal, by inclusion, non\-emp\-ty lamination is called a \emph{minimal lamination}.
A \emph{minimal sublamination of $\lam\ne \lam_\0$} is a sublamination
of $\lam$ that is minimal (and hence nonempty).
\end{dfn}

Consider two examples of quadratic minimal laminations.
First, take the lamination $\lam_P$ associated with a polynomial
$P(z)=e^{2\pi i\alpha} z+z^2$ with $\alpha\in\Q$.
Then $J_P$ is locally connected, and $\lam_P$ looks as explained
 in Subsection \ref{ss:lclam}, with an invariant gap-leaf representing the fixed parabolic point $0$.
Moreover, $\lam_P$ is minimal as it consists of the grand orbit of one leaf
(namely, of any edge of the nondegenerate invariant gap-leaf of $P$).

As another quadratic example,
 take a non-horizontal $\si_2$-critical leaf $\ol{xy}$, where $\si_2^N(x)=\si_2^N(y)=a_0$ is the unique
$\si_2$-fixed point of $\uc$ (i.e., the point with argument $0$) for some positive integer $N$,
consider iterated pullbacks of $\ol{xy}$ compatible with $\ol{xy}$, and close this %the resulting
set of chords to obtain a lamination $\lam$
(the construction of a \emph{pullback lamination} is due to Thurston \cite{thu85}).

Assume that there exists a nonempty $\lam'\subsetneqq \lam$;
then $\ol{xy}$ is contained in a gap $G$ of $\lam'$ as otherwise
$\lam'$ contains all pullbacks of $\ol{xy}$ and the closure of their union, i.e. $\lam'=\lam$, a contradiction.
The gap $G$ cannot contain $a_0$ as then it must be an invariant gap of $\lam'$ that maps onto itself two-to-one
which implies that $G=\ol{\disk}$ and that $\lam'$ is degenerate, a contradiction.

Consider $\si_2^N(G)$. Observe that $\si_2^N(x)=a_0\in \si_2^N(G)\cap \uc$. Also,
$\si_2(G)$ is not a point, and since no further image of $\si_2(G)$ equals $\ol{xy}$
(recall that $\ol{xy}$ is inside a gap $G$ of $\lam'$ and, hence, is not a leaf of $\lam'$)
then $\si_2^N(G)$ is not a point, and $a_0\in \si_2^N(G)$.
Since $\ol{xy}$ is a leaf of $\lam$, all points of $\si_2^N(G)\cap \uc$ have orbits contained in
the half-circle $S_0$ with endpoints $x, y$ containing $a_0$, and the only such point
is $a_0$ itself, a contradiction. Hence $\lam$ is minimal.

We claim that $\lam$ is perfect.
Indeed, repeatedly pulling $\ol{xy}$ back towards $a_0$ one can find a leaf $\ell$ of $\lam$ that is
arbitrarily close to $a_0$ and separates $a_0$ from $\ol{xy}$.
Since $\si_2^N(\ol{xy})=a_0$, we can pull $\ell$ back $N$ steps
 along the backward orbit of $a_0$ that leads to $\ol{xy}$.
In this way, one obtains two leaves with the same images enclosing $\ol{xy}$ in a narrow strip.
Thus, $\ol{xy}$ is not isolated in $\lam$. Hence
pullbacks of $\ol{xy}$ are not isolated in $\lam$ either. Since by definition any leaf of $\lam$ is either a pullback of $\ol{xy}$
or a limit of such pullbacks, $\lam$ is perfect.

A maximal, by inclusion, perfect sublamination $\lam^p$ of $\lam$ is called the \emph{perfect part} of $\lam$;
it is the set of all leaves $\ell\in \lam$ such that, arbitrarily close to $\ell$, there are
uncountably many leaves of $\lam$. %Perfect laminations are clean.

\begin{lem}\label{l:perf-count}
If $\lam$ is an invariant lamination, then so it $\lam^p$.
If $\lam$ is uncountable, then $\lam^p\subset\lam$ is nonempty.
A minimal lamination is either perfect or countable.
In the latter case all its nondegenerate leaves are isolated.
\end{lem}

\begin{proof} By \cite[Lemma 3.12]{bopt20},
the set $\lam^p$ is an invariant lamination. If $\lam$ is uncountable,
then it is easy to see that
$\lam^p\subset\lam$ is nonempty. The last claim holds by part (2) of Lemma \ref{l:clos-sibl}.
\end{proof}

\begin{lem}\label{l:simple}
If $\lam$ is nonempty, then $\lam$ contains a minimal lamination.
\end{lem}

\begin{proof}
Let $\lam_\al$ be a nested family of laminations.
Definition \ref{d:sibli} implies
that then $\bigcap \lam_\al$ is a sibling invariant lamination too.
If all $\lam_\al$ are nonempty, then by Lemma \ref{l:compute} each of them has a leaf of length at least $\frac{1}{d+1}$ and so
$\bigcap \lam_\al$ is nonempty. Now the desired statement follows
from Zorn's lemma.
\end{proof}

For $\lam$ and a nondegenerate leaf $\ell\in \lam$, let $\mathcal
G(\ell)\subset \lam$ be the set of all ite\-rated pullbacks of $\ell$ and of all its nondegenerate iterated
images. Lemma \ref{l:perf-count} and compactness of invariant laminations imply Lemma \ref{l:dense}.

\begin{lem}\label{l:dense}
Let $\lam$ be a minimal lamination. If $\ell\in \lam$ is a nondegenerate
leaf, then all nondegenerate leaves of $\lam$ are in the closure of $\mathcal G(\ell)$.
In particular, if $\lam$ is minimal and countable, then for any nondegenerate leaf $\ell$ of
$\lam$ we have that $\mathcal G(\ell)=\lam$.
\end{lem}

\subsection{Invariant objects}
\label{ss:invobj}

Let $\Delta\subset\C$ be an open Jordan disk.
Recall \cite[Definition 3.6]{gm93} that a continuous map $f:\ol\Delta\to\C$ is \emph{weakly polynomial-like}
 (\emph{weakly PL} for short) of degree $d$ if $f(\bd(\Delta))\cap\Delta=\0$, and the induced map
 on integer homology
$$
f_*:H_2(\ol\Delta,\bd(\Delta))\cong\Z\ \to\  H_2(\C,\C\sm\{z_0\})\cong \Z
$$
is the multiplication by $d$, where $z_0$ is any base point in $\Delta$.

\begin{lem}\cite[Lemma 3.7]{gm93}
  \label{l:fxptgm}
  If $f:\ol\disk\to\C$ is weakly PL with isolated fixed points, then
  the degree of $f$ equals the sum of the Lefschetz indices over all fixed points of $f$ in $\ol\disk$.
\end{lem}

The notion of the \emph{Lefschetz index}, adapted to the present context, is discussed in \cite{gm93}.
A gap $G$ of a lamination is
\emph{invariant} if $\si_d(G)=G$ (with ``$=$'' rather than ``$\subset$'').

\begin{lem}
  \label{l:inv-exist}
A $\si_d$-invariant lamination has an invariant gap-leaf or an invariant infinite gap.
\end{lem}

\begin{proof}
  Consider a $\si_d$-invariant lamination $\lam$.
Extend $\si_d$ linearly over each simplex in the barycentric subdivision of $\lam$ (cf. \cite{thu85}),
 and denote the extended map by $f:\ol\disk\to\C$; $f$ is weakly PL.
If there are invariant nondegenerate leaves of $\lam$ or invariant gaps with
 fixed points on the boundary, we are done.
Otherwise, $f$ has isolated fixed points, and there are
$d-1$ fixed points on $\uc$, all with Lefschetz index one.
Thus, there is a fixed point $a$ inside $\disk$.
If $G$ is a gap of $\lam$ containing $a$, then $G$ is clearly invariant.
\end{proof}

\section{Invariant gaps, their canonical laminations, and flower-like sets}
\label{s:invgap}

By \emph{cubic} (resp., \emph{quadratic}) laminations, we always mean \emph{sibling $\si_3$-}(resp.,
\emph{$\si_2$-}) \emph{invariant} laminations.
For brevity we denote the horizontal diameter of the unit disk by $\hdi$.
\textbf{\emph{From now on $\lam$ (possibly with
sub- and superscripts) denotes a cubic sibling invariant
(pre)lamination.}} This section is based upon \cite{bopt16} but contains some further developments, too.

\subsection{Invariant gaps} An \emph{invariant gap} is an invariant gap of a cubic lamination
 (the latter may be unspecified). Various types of infinite invariant gaps are described in \cite{bopt16}.

 \subsubsection{Quadratic invariant gaps}\label{sss:quig}
 An infinite invariant gap is \emph{quadratic} if it has degree 2.
A critical chord $\oc$ gives rise to an open circle arc
$L(\oc)$ of length $2/3$ with the same endpoints as $\oc$;
let $I(\oc)$ be the complement of $\ol{L(\oc)}$.
Clearly, the set $\Pi(\oc)$ of all points with orbits in $\ol{L(\oc)}$ is nonempty,
closed and forward invariant. Let $\Pi'(\oc)$ be the maximal perfect subset of $\Pi(\oc)$.
Then by \cite[Lemmas 3.3,3.6,3.9]{bopt16} the convex hulls $G(\oc)$ of \newline $\Pi(\oc)$  and  $G'(\oc)$ of $\Pi'(\oc)$ are
invariant quadratic gaps, and any invariant quadratic gap is like that.
If $\oc$ has non-periodic endpoints, or both its endpoints eventually map to $I(\oc)$,
then $\Pi(\oc)=\Pi'(\oc)$ and $G(\oc)=G'(\oc)$. This allows us to classify quadratic invariant gaps.

\begin{dfn}\label{d:class3}
If the $\si_3$-orbit of $\si_3(\oc)$ is contained in $ L(\oc)$,
then $\oc$ is an edge of $G(\oc)$, and $G(\oc)$  is said to be of \emph{regular critical type}.
If an endpoint  of $\oc$ is periodic, and the orbit of $\oc$ is contained in $\ol{L(\oc)}$,
then $\Pi'(\oc)\subsetneqq \Pi(\oc)$, the gap
$G(\oc)$ is said to be of \emph{caterpillar} type, and
$G'(\oc)$ is said to be of \emph{periodic} type.
\end{dfn}

Quadratic invariant gaps give rise to \emph{canonical} laminations
\cite[Lemmas 3.11, 3.12 and 3.13]{bopt16}.
Defining canonical laminations is easy for quadratic invariant gaps of regular critical and caterpillar type as in those cases
 (i.e., if the major of the quadratic invariant gap $U$ is critical)
 there is a \emph{unique} lamination $\lam_U$ that has the gap $U$.
Indeed, under the given assumptions, we already have two critical sets of $\lam_U$, namely $U$ and its
(critical) major, therefore,
 all iterated pullbacks of $U$ that are gaps of $\lam_U$ are well-defined;
 these pullbacks of $U$ ``tile'' the entire $\cdisk$ and give rise to $\lam_U$.

For a gap $U$ of periodic type, $\lam_U$ is as follows.
Add to $U$ a critical quadrilateral $Q_U=Q$, which is the convex hull of
 the major $M_U=M$ of $U$ and its sibling $M'_U=M'$ located outside of $U$.
Form the pullback lamination with critical sets $U$ and $Q$;
 this lamination is now well defined.
After removing the edges $\ell$ and $\ell'$ of $Q$ distinct from $M$, $M'$, and all
iterated pullbacks of $\ell$ and $\ell'$,
 we obtain the canonical lamination $\lam_U$ which,
 unlike before, has \emph{two} cycles of Fatou gaps.
Indeed, $\lam_U$ has a gap $V\supset Q$ with edges $M$ and $M'$,
 and $\si_3|_{\bd(V)}$ is two-to-one.
The entire $\cdisk$ is ``tiled'' by concatenated pullbacks of $U$ and $V$.
In this context, $U$ is said to be the \emph{senior (gap)}, and $V$ is called the \emph{vassal (gap)}, cf.
\cite[Lemma 3.4]{bopt16}.

\subsubsection{Invariant gaps of degree one}

Finite rotational sets (under the name of fixed point portraits) are classified in
\cite{gm93}; in \cite[Subsection 4.1]{bopt16}, we specify the picture for the cubic case.
Infinite invariant gaps of degree 1 are studied in
\cite[Lemma 4.6]{bopt16} and in Lemma \ref{l:crit-must}. Lemma \ref{l:invagap} follows from these results.

\begin{lem}\label{l:invagap}
A degree one invariant gap $G$ of a cubic lamination has one or two majors;
every edge of $G$ eventually maps to a major and, if $G$ is infinite,
at least one of its majors is critical.
\end{lem}

An invariant gap $G$ is \emph{rotational} if $\si_3$ acts on
$G\cap\uc$ as a combinatorial rotation different from the identity.
A chord is \emph{compatible} with a finite collection
of gaps if it does not cross edges of these gaps.

\subsection{Flower-like sets} Let us introduce the following useful concept.

\begin{dfn}[Flower-like sets]\label{d:flower}
Suppose that $\lam$ has an infinite invariant gap $U$ or an invariant
gap-leaf $G$ and an infinite gap $U$ that shares an edge with $G$
(in the latter case, it follows that $U$ is periodic).
Then $\{U\}$ (in the former case) or the set consisting of $G$ and all periodic Fatou gaps attached to it
(in the latter case) is said to be a \emph{flower-like set}. Thus, a flower-like set
is a certain set of gaps or gap-leaves. Flower-like sets can be viewed as standing alone
(i.e., without specifying a lamination but with the understanding that such a lamination exists).
\end{dfn}

By saying that $\lam$ \emph{has} a flower-like set $F$, we mean that all
gaps/gap-leaves from $F$ are gaps/gap-leaves of $\lam$.
Say that $\lam$ is \emph{compatible} with a flower-like set $F$ if no leaf of $\lam$ crosses an edge of a gap from $F$.
Flower-like sets represent dynamics of polynomials with a non-repelling fixed point.

\begin{lem}\label{l:flowerc}
If sets $F_i, i=1, 2, \dots$ are flower-like, %sets,
 and $F_i\to F$, then $F$ contains a flower-like set.
\end{lem}

\begin{proof}
If periodic gaps/gap-leaves $H_i$ of period $m$ converge to a gap/gap-leaf
$H$, then $H$ is periodic of period $m$. We claim that if
all $H_i$'s are infinite, then $H$ must also be infinite.

Indeed, let $H_i$ be an infinite gap
of period $m$. Then there exists an edge $\ell_0$ of $H_i$ such that a backward orbit $\Ac$ of $\ell_0$
formed by edges of gaps from the orbit of $H_i$ consists of infinitely many distinct leaves.
For each leaf $\ell_j=\ol{x_jy_j}\in \Ac$ which is an edge of a gap $V_j$ from the orbit of $H_i$
let $t(\ell_i)$ be the length of the circle arc with endpoints $x_j$ and $y_j$
which contains no vertices of $V_j$. Evidently (see Lemma \ref{l:compute}) we can choose $\ell$ so that
$t(\ell)\ge \frac{1}{d+1}$ while other numbers $t_i$ are less than $\frac{1}{d+1}$. Recall that when we pull
back the length of the associated arc decreases \emph{at most} $d$-fold. This implies
that for any integer $N>0$, there is $\varepsilon>0$ such that
$N$ edges of $H_i$ are longer than $\varepsilon$ for any $i$. Choosing a subsequence, we see that $H$ has
at least $N$ nondegenerate edges. Repeating this argument for every $N$, we see that $H$ is an infinite gap.

In particular a sequence of invariant infinite gaps converges to an invariant infinite gap. Hence, if $F_i$'s
are invariant infinite gaps, then we are done.

Thus, we may assume that each $F_i$ consists of a finite invariant gap-leaf
$G_i$ and the cycle of an infinite gap $U_i$ sharing
an edge $\ell_i$ with $G_i$, and that $G_i$'s converge to an invariant gap $G$.
If $G$ is infinite, we are done. Assume that $G$ is finite.
Then $G_i=G$ for large $i$ (by Lemma \ref{l:limleaf}), and
by the above $U_i\to U$ where $U$ is infinite and periodic, which completes the proof.
\end{proof}

\section{Two types of cubic minimal laminations}\label{s:2types}

We classify cubic minimal laminations into central and non-central ones.

\begin{dfn}\label{d:central}
A minimal lamination is \emph{central} if it is compatible with a flower-like set;
a minimal lamination is \emph{non-central} otherwise.
\end{dfn}

If $A\subset \uc$ is a closed set and $a, b\in A$ then a chord $\ol{ab}$ of $\uc$
is called a \emph{diagonal} of the convex hull $\ch(A)$ if intersects the interior of $\ch(A)$.

Recall that for a lamination $\lam$ and a nondegenerate leaf $\ell\in \lam$, $\mathcal
G(\ell)\subset \lam$ is the set of all ite\-rated pullbacks of $\ell$ and of all its nondegenerate iterated
images.

\begin{lem}\label{l:count0}
Let $\lam$ be a cubic countable minimal lamination. Then:
\begin{enumerate}
  \item all nondegenerate leaves of $\lam$ are isolated;

  \item for any nondegenerate leaf $\ell\in \lam$, the set of all
      nondegenerate leaves in $\lam$ coincides with $\mathcal
      G(\ell)$;

  \item the lamination $\lam$ has a flower-like set.

\end{enumerate}

\end{lem}

\begin{proof}
(1) This claim follows from Lemma  \ref{l:perf-count}.

(2) By Lemma \ref{l:dense}, the set $\mathcal G(\ell)$ is dense in $\lam$. Since each leaf
is of $\lam$ is isolated, then $\mathcal G(\ell)$ equals the set of all
      nondegenerate leaves in $\lam$.

(3) By Lemma \ref{l:inv-exist}, find an invariant gap-leaf or infinite gap $G$ of $\lam$.
If $G$ is infinite, then by definition $\lam$ has a flower-like set. Assume that $G$ is finite.
Let $\ell$ be an edge of $G$; it is isolated by (1).
Let $H$ be a gap of $\lam$ attached to $G$ along $\ell$.
If $H$ is infinite,  then, again by definition, $\lam$ has a flower-like set.
Assume that $H$ is finite. If $n$ is the period of $\ell$, then there are two cases: $\si_3^n(H)=H$ and $\si_3^n(H)=\ell$.
The former case contradicts (2), hence we may assume that $\si_3^n(H)=\ell$. If we follow the orbit of
$H$ we then have to find a moment at which $H$ will be mapped to the edge of $G$ (it is possible, that
this will not happen right away, but, since $\si_3^n(H)=\ell$, it must happen at some point). Thus, without loss of generality
we may assume that $H$ is a gap, $H\cap G=\ell$ is a leaf and $\si_3(H)=\si_3(\ell)$.

Then there are several cases listed below. This is a complete list because of the properties of laminations and the
fact that the degree of the map is 3.

a) The gap $H$ is a quadrilateral with a sibling gap of $G$ attached to $H$ at the edge
of $H$ that is itself a sibling of $\ell$. Then there are edges $\oq_1$ and
$\oq_2$ of $H$ that share endpoints with $\ell$. They are isolated in $\lam$.
It follows by Definition \ref{d:sibli} of a sibling invariant lamination that we can remove these leaves and their pullbacks
and still have a smaller sibling invariant lamination than $\lam$, a contradiction with the fact that $\lam$ is a minimal lamination.

b) The gap $H$ is a hexagon with
two more edges $\ell_1$ and $\ell_2$ such that $\ell,$ $\ell_1$ and $\ell_2$ are pairwise disjoint sibling leaves
and at $\ell_1$ and $\ell_2$ sibling gaps of $G$ are attached to $H$ (the fact that $\ell,$ $\ell_1$
and $\ell_2$ are pairwise disjoint uniquely defines $\ell_1$ and $\ell_2$ since $\ell$ is given).
In this case, similar to case (a), we can remove the remaining three edges of $H$ and all their pullbacks and still
have a smaller sibling invariant lamination than $\lam$, a contradiction with the fact that $\lam$ is a minimal lamination.

We conclude that neither case (a) nor case (b) is possible under the assumption about minimality of $\lam$. This finally rules out
the case when $H$ is finite and proves that $\lam$ has a flower-like set as claimed.
\end{proof}

There are also uncountable (hence, by Lemma \ref{l:perf-count}, perfect)
minimal laminations compatible with flower-like sets, but not having them.
Here is a heuristic example. Take a perfect non-renormalizable
quadratic lamination $\lam_2$ with critical diameter $\ell$. Choose a non-periodic non-precritical leaf $\ol{xy}\in \lam_2$ and
blow up $y$ to create a regular $\si_3$-critical major $\oy$ which results into a triangle contained in a quadratic invariant gap $U$ of $\si_3$, then
reflect this triangle with respect to $\oy$ and erase $\oy$ to create a critical quadrilateral $Q$. Together with the leaf $\hell$ that
used to be $\ell$ before the transformations, we have two critical sets $Q$ and $\hell$ that define a cubic lamination $\lam_3$.
We claim that this is a perfect minimal lamination. Indeed, suppose that there is a nonempty lamination $\hlam\subsetneqq \lam_3$. If
$\hell\notin \hlam$, then $\hell$ is inside a gap $G$ of $\hlam$. By the assumptions pullbacks of $\hell$ in $\lam_3$ approach
$\hell$ from both sides; hence $G$ cannot be finite and so $G$ must be infinite, a contradiction with the fact that $\lam$ is non-renormalizable.
Since pullbacks of $\hell$ approximate $\ol{xy}$, the set $Q$ survives, too, and in the end $\lam_3=\hlam$, a contradiction. By construction,
$\lam_3$ is compatible with $U$ while edges of $U$ are not leaves of $\lam_3$.

\begin{comment}

We will need the following result which relies upon Lemma \ref{l:pisp}
and is a rather a small part of Theorem 3.8 \cite{bopt13}.

\begin{lem}\cite[Part of Theorem 3.8]{bopt13}\label{l:part3.8}
A perfect lamination has infinitely many periodic gap-leaves.
\end{lem}

\end{comment}

\begin{lem}\label{l:periodense}
A non-central minimal lamination $\lam$ is perfect and has infinitely many periodic gap-leaves.
Given any periodic leaf of $\lam$, the family of all iterated pullbacks of it is dense in $\lam$.
\end{lem}

\begin{proof}
By Lemma \ref{l:perf-count}, a minimal lamination is countable or perfect. By Lemma \ref{l:count0} and since
$\lam$ is non-central, $\lam$ is perfect. %, and, hence, clean.
By Lemma \ref{l:pisp} $\lam$ is proper. By Theorem \ref{t:proper} this gives rise to a
laminational equivalence relation $\approx_\lam$ and the corresponding topological polynomial
$f_{\approx_\lam}=f_\lam:J_\lam\to J_\lam$ to which
$\si_3$ is semiconjugate by a map $\varphi$. Since $\lam$ is perfect, there are uncountably many grand orbits of nondegenerate
leaves of $\lam$ containing no leaves of critical sets of $\lam$.
If $\ell$ is a leaf from such grand orbit, then $\varphi(\ell)=x$ is a cutpoint of $J_\lam$, and
all points of the $f_\lam$-orbit of $x$ are cutpoints of $J_\lam$.
Such dynamics was studied in \cite{bopt13} where, in Theorem 3.8, it was proven that $f_\lam$ has infinitely many
periodic cutpoints.
Taking their $\varphi$-preimages, we see that $\lam$ has infinitely many periodic
 gap-leaves. The last claim holds by Lemma \ref{l:dense}.
\end{proof}

%It remains to consider central minimal perfect laminations.

Let us now consider central minimal perfect laminations.

\begin{lem}\label{l:flower-prop}
If a perfect minimal lamination $\lam^{min}$ is central, then
there exists a flower-like set $E$ without caterpillar or Siegel gaps
compatible with $\lam^{min}$. Moreover, there exists
a minimal proper lamination compatible with $\lam^{min}$ that has the flower-like set $E$ and has no Siegel or caterpillar gaps.
\end{lem}

\begin{proof}
By definition $\lam^{min}$ is compatible with a flower-like set $F$. Consider cases.
Let $F$ be an invariant gap $U$. Since $\lam^{min}$ is perfect,
the existence of a diagonal of $U$ that is a leaf of $\lam^{min}$ implies the existence of uncountably many such diagonals.
If $U$ is a Siegel gap, then eventual images of these diagonals will cross one another, a contradiction.
On the other hand, if $U$ is a caterpillar gap of degree one, then it has only countably many vertices, a contradiction.
Hence an infinite invariant gap $U$ of degree one is always contained in an infinite invariant gap $V$ of $\lam^{min}$.
Since by Lemma \ref{l:crit-must} perfect laminations have no periodic caterpillar or Siegel gaps, $V$ is a quadratic non-caterpillar gap.
Set $E=\{V\}$.

Let $F$ be a flower-like set but not an invariant infinite gap. It suffices to consider
the case when $F$ involves periodic Fatou gaps whose boundaries include maximal (countable)
concatenations of edges. A concatenation $A$ like that always connects two points $a$ and $b$. We claim that
then $\ol{ab}$ is compatible with $\lam^{min}$. Indeed, otherwise there must exist a leaf $\ell$
of $\lam^{min}$ that crosses $\ol{ab}$.
However, this contradicts the assumption that $\lam^{min}$ is perfect.
Now, each such concatenation $A$ of edges can be replaced with the corresponding chord $\ol{ab}$.
In this way, one obtains a new flower-like set $E$ with no nontrivial concatenations of edges,
 which implies that $E$ is a flower-like set that does not include caterpillar gaps.

Consider all pullbacks of $E$ consistent with critical sets that are parts of $E$, and the critical sets of $\lam^{min}$.
Since $E$ is forward invariant (in fact, it maps onto itself under $\si_3$),
the closure $\lam_E$ of this system of pullbacks is an invariant lamination,
and, by construction, it is compatible with $\lam^{min}$ (the details about the %construction of
pullback laminations can be found in \cite{thu85}).
Moreover, since the construction already uses the critical sets of $\lam^{min}$ and %as well as
critical sets contained in $E$,
it cannot give rise to critical leaves with periodic endpoints or critical wedges with periodic endpoints.
%(if a lamination is not proper, then it  has isolated leaves).
Thus, $\lam_E$ is proper. Therefore, it cannot contain caterpillar gaps. Moreover, since the construction starts
with the set $E$ that does not include Siegel gaps, it follows that it will not lead to Siegel gaps. Thus,
$\lam_E$ is a proper lamination with no Siegel or caterpillar gaps.
%. Moreover, it has no Siegel or caterpillar gaps as these would have isolated edges.
It remains to take a minimal sublamination $\hlam$ of $\lam_E$.
Evidently, $\hlam$ will contain a flower-like set,
has no caterpillar or Siegel gaps, and is compatible with the lamination $\lam^{min}$.
\end{proof}

The next lemma deals with unions of laminations.

\begin{lem}\label{l:unite}
Let $\sim_1$ and $\sim_2$ be $\si_3$-invariant laminational
equivalence relations with associated laminations $\lam_1$ and $\lam_2$
such that $\lam_1$ and $\lam_2$ are compatible. Then the following holds.

\begin{enumerate}

\item $\lam_1\cup \lam_2=\lam$ is a proper lamination.

\item If $\lam_1$ is perfect and $\lam_2$ has no Siegel gaps, then $\lam$ has no Siegel or caterpillar gaps.

\end{enumerate}

In particular, $\lam_1, \lam_2$ and $\lam$ belong to the same T-class of laminations.
\end{lem}

\begin{proof}
Note that $\lam_1$ and $\lam_2$ are proper laminations.
Claim (1) of the lemma follows from the definitions and Theorem \ref{t:proper}.
Indeed, it is clear that $\lam_1\cup \lam_2$ has no critical leaves
with periodic endpoints. Suppose that $\lam_1\cup \lam_2$ has a
a critical wedge $\ol{xp}\cup \ol{py}$ with an $n$-periodic vertex $p$. Then
at least one of the points $x, y$ is non-periodic, a contradiction with Theorem \ref{t:proper}.

Consider now claim (2). Neither $\lam_1$ nor $\lam_2$ has a Siegel gap.
Suppose that $\lam_1\cup \lam_2$ has a periodic Siegel gap $H$ with a critical leaf $\hell$;
then $H$ must be the intersection of a periodic infinite gap $U_1$ of $\lam_1$ and
a periodic infinite gap $U_2$ of $\lam_2$. We may assume that an edge $\ell_1$ of $U_1$ is an edge of $H$ and a diagonal of $U_2$,
and that an edge $\ell_2$ of $U_2$ is an edge of $H$ and a diagonal of $U_1$.
Suppose that $\lam_1$ is perfect. Since every edge of a Siegel gap eventually maps to a critical leaf it follows that
an image of $U_1$ has a critical edge $\ell=\si_3^n(\ell_1)$ for some $n\ge 0$.
However, $\ell$ is then a critical edge of an infinite gap $\si_3^n(U_1)$ of $\lam_1$,
which implies that $\ell$ is isolated in $\lam_1$, a contradiction. Also, since $\lam$ is proper, it has no caterpillar gaps.
The last claim of the lemma follows from the definitions.
\end{proof}

Lemma \ref{l:no-perf-cent} deals with perfect minimal laminations with flower-like set.

\begin{lem}\label{l:no-perf-cent}
A perfect minimal laminations with flower-like set must have a quadratic invariant gap of periodic type.
\end{lem}

\begin{proof}
Suppose that $\lam$ is a perfect minimal lamination with a flower-like set $F$. If $F$ includes an invariant finite gap $G$ and
an infinite periodic gap $U$ attached to $G$ then it (and, hence, $\lam$) has isolated leaves, a contradiction. Suppose that $F$
includes an invariant gap $U$; since $\lam$ is perfect, $U$ must be quadratic of periodic type as any other type forces the existence
of a critical  major of $U$ which will be isolated.
\end{proof}

\section{Fibers and alliances}\label{s:fiballi}

For sets $A$, $B$, let $A\vee B$ be the set of all \emph{unordered} pairs $\{a,b\}$ with $a\in A, b\in B$.
Let $\cch$ be the set of all $\si_3$-critical chords with the natural
topology; $\cch$ is homeomorphic to $\uc$. Consider cubic critical portraits.
Recall that \emph{distinct} chords \emph{cross} if they have common points in $\disk$.

\begin{dfn}[Cubic critical portraits]\label{d:cripo} A (cubic) \emph{critical portrait}
is a pair $\{\oc,\oy\}\in\cch\vee\cch$ such that $\oc$ and $\oy$ do not
cross (in particular, do not coincide).
Let $\crp=\crp_3$ be the space of all cubic critical portraits.
Topology on $\crp$ is defined as that induced from $\uc\vee \uc$, that is,
 the topology on the set of all unordered pairs of points in $\uc$
 given by the Hausdorff distance.
With this topology, $\crp$ is compact and Hausdorff (note that
 two chords crossing is an open condition).
For an invariant cubic $\lam$, let $\crp(\lam)$ be the family of all critical
portraits compatible with $\lam$; if $\K\in \crp(\lam)$, call
$\K$  a critical portrait \emph{of} $\lam$.
\end{dfn}

The space $\crp$ is homeomorphic to the M\"obius band (see, e.g., \cite{tby20}).

\begin{lem}\cite[Lemma 3.53]{bopt20}
\label{l:3.53}
If there is a critical portrait compatible with cubic laminations $\lam$ and $\lam'$,
 then any leaf of $\lam$ crosses at most countably many leaves of $\lam'$, and vice versa.
\end{lem}

We are ready to construct upper-semicontinuous partitions of $\Cc_3$ and $\crp$
into subsets resulting in two homeomorphic quotient spaces. This gives rise to a model
of $\Cc_3$. The elements of these
partitions will be called \emph{fibers}. The collections of laminations associated with such
fibers are instrumental in the arguments, but are not used in the Main Theorem.
To emphasize that, we call such collections of laminations \emph{alliances} rather than fibers.
Recall that, unless stated otherwise, we consider cubic invariant laminations.

\begin{dfn}[Fibers of critical portraits and alliances of laminations]\label{d:alliance}
\,\newline
\indent Below, we define \emph{fibers} of critical portraits and
\emph{alliances} of laminations.
\begin{enumerate}

\item The set $\crp(\lam)$ of all critical portraits compatible with a non-central minimal lamination
$\lam$ is called a \emph{non-central fiber of critical portraits (generated by $\lam$)}.
Critical portraits from non-central fibers of critical portraits are said to be \emph{non-central}.

\item The set $\F_0$ of all critical portraits compatible with central minimal laminations is called the
\emph{central fiber of critical portraits}. Critical portraits from $\F_0$ are called \emph{central}.

\item All laminations whose minimal laminations are non-central are said to be \emph{non-central}.
Non-central laminations with the same minimal lamination form a \emph{non-central alliance of laminations}.
The empty lamination and all laminations with central minimal laminations are called \emph{central}.
All central laminations form the \emph{central alliance of laminations} $\Ac_0$.

\end{enumerate}

\end{dfn}

For example, consider the critical portrait $\K:=\{\oc=\ol{0 \frac13}, \oy=\ol{\frac13 \frac23}\}$. It is easy to see that it is
compatible with a central minimal lamination (in fact, with more than one). Indeed, consider the lamination that has
an invariant Fatou gap with $0$ on its boundary and with the edge $\oy$; it is easy to see that it is a minimal central lamination
(all its non-degenerate leaves are pullbacks of $\oy$) and that it is compatible with $\K$.

%We are ready to prove a key lemma dealing with non-central minimal laminations.

\begin{lem}
\label{l:crplam}
Let $\lam$ be non-central and minimal, and $\lam'$ be minimal.
Let $\K=\{\oc,\oy\}$ be a critical portrait compatible with $\lam$ and $\lam'$. Then $\lam'=\lam$.
%then, for any minimal lamination $\lam'$, we have that $\crp(\lam)\cap \crp(\lam')\ne \0$ implies $\lam'=\lam$.
\end{lem}

Recall that a lamination
is non-central if it is not compatible with a flower-like set.

\begin{proof}
For clarity we divide
almost all the proof into
steps on each of which we prove a certain claim placed in the beginning of each paragraph.

(1) \emph{$\lam$ is perfect:} since $\lam$ is non-central, it follows from Lemma \ref{l:periodense}.

(2) \emph{$\lam$ and $\lam'$ are compatible:} otherwise the fact that $\lam$ is perfect implies that
a leaf of $\lam'$ would cross uncountably many leaves of $\lam$ contradicting Lemma \ref{l:3.53}.

%(3) \emph{$\lam'$ is perfect:} by Lemma \ref{l:count0} it follows from the fact that $\lam'$ does not have a flower-like set as otherwise
%$\lam$ is compatible with it, a contradiction with $\lam$ being non-central.

(3) \emph{$\lam'$ is perfect:} suppose that $\lam'$ is countable. Then by Lemma \ref{l:count0} (3) it has a flower-like set.
By (2) this implies that $\lam$ is compatible with a flower-like set, a contradiction with $\lam$ being non-central.

(4) \emph{$\lam$ and $\lam'$ have a common invariant gap-leaf $G$.} Indeed, by Lemma \ref{l:inv-exist}, there is an invariant
gap-leaf or infinite gap $G'$ of $\lam'$. Since $\lam$ is non-central and compatible with $G'$, then
$G'$ is in fact a gap-leaf.
If a leaf $\ell\in \lam$ is inside $G'$, then, since $\lam$ is perfect, other leaves of $\lam$ approximate $\ell$
and cross leaves of $\lam'$, a contradiction. Hence $G'\subset G$, where $G$ is an invariant
gap-leaf of $\lam$ (since $\lam$ is non-central, $G$ is finite). Since no leaf of $\lam'$ can be inside $G$, then $G=G'$.

(4) \emph{Each edge of $G$ is a limit of leaves of $\lam$, and a limit of leaves of $\lam'$:} follows from the fact
that both $\lam$ and $\lam'$ are perfect.

The idea of what follows is to compare iterated pullbacks of $G$ with respect to $\lam$ and $\lam'$
and rely upon Lemma \ref{l:dense}.

(5) \emph{If iterated images of $\oc$ and $\oy$ avoid $G$, then $\lam=\lam'$:}
indeed, in the case at hand iterated $\lam$-pullbacks of $G$ and iterated $\lam'$-pullbacks
of $G'$ are the same. By Lemma \ref{l:dense} they are dense in both $\lam$ and $\lam'$
which implies the claim.

Consider now the only remaining case when for some minimal $n\ge 0$ the $n$-th
image of a critical leaf from $\K$ is a vertex of $G$. For the sake of definiteness
assume that the point $\si_3^n(\oc)$ is a vertex of $G$.
Let $C$ be the critical set of $\lam$ containing $\oc$; then $C$ is the $\si_3^n$-pullback of $G$
in the lamination $\lam$ (observe that $\lam'$ is perfect and cannot have a Fatou gap
attached to an invariant gap-leaf).
Similarly, let $C'$ be critical set of $\lam'$ containing $\oc$; then $C'$ is the $\si_3^n$-pullback of $G$
in the lamination $\lam'$.

(6) \emph{We claim that $C=C'$:} by the choice of $n$ all pullbacks of $G$ in the sense of both $\lam$ and
$\lam'$ coincide through the $\si_3^{n-1}$-st pullback $A$. On the next step $A$ pulls back to $\oc$ itself,
and there are the following possibilities.

(i) $A$ pulls back to its 2-to-1 immediate preimage containing $\oc$ as a diagonal
and contained in the union of the closures of two components of $\cdisk\sm \oc$ that border on $\oc$.

(ii) $A$ pulls back to its 3-to-1 immediate preimage containing $\oc$.

However if both possibilities realize then edges of one of the sets $C,$ $C'$ will be diagonals of the other set,
and as we saw above this is impossible because both laminations $\lam$ and $\lam'$ are perfect. Thus, $C=C'$.

(7) \emph{All pullbacks of $G$ in both laminations $\lam$ and $\lam'$ are the same.} Indeed,
if case (ii) realizes, then $C=C'$ contains both $\oc$ and $\oy$ as its diagonals and from this moment on all its
pullbacks are well-defined and coinciding in both $\lam$ and $\lam'$. This exhausts all iterated pullbacks of $G$.
If case (i) realizes, then, again, pullbacks of $C=C'$ for both $\lam$ and $\lam'$ coincide until one of the pullbacks
hits $\oy$. However, in this case, too, the corresponding pullback of $C=C'$ to $\oy$ in both $\lam$ and $\lam'$ is the same because
it must be critical and $2$-to-$1$ in our case. From that moment on the pullbacks of $G$ are unique. Thus,
all pullbacks of $G$ in both laminations $\lam$ and $\lam'$ are the same.

Lemma \ref{l:dense} now implies the claim of the lemma.
\end{proof}

Theorem \ref{t:no-share-1} shows that fibers of critical portraits are pairwise disjoint. Moreover, alliances of laminations
are pairwise disjoint, too.

\begin{thm}\label{t:no-share-1}
Let $\lam\ne \lam'$ be minimal laminations.
Then
$$\crp(\lam)\cap \crp(\lam')=\0$$ unless $\lam$ and $\lam'$ are central (i.e., a minimal non-central lamination is not compatible with any other minimal lamination).
Thus, distinct fibers of critical portraits are disjoint, and distinct alliances of laminations are disjoint.
Moreover, any non-central alliance of laminations is a T-class of laminations while the central alliance of laminations is the union of
all T-classes of central laminations.
\end{thm}

\begin{proof}%[Proof of Theorem \ref{t:no-share-1}]
Consider two distinct fibers of critical portraits. As they are distinct, one of them
is generated by a minimal non-central lamination $\lam$. By Lemma \ref{l:crplam},
if the other fiber intersects $\crp(\lam)$, then it must coincide with $\crp(\lam)$, a contradiction. Thus, the fibers in question are disjoint.

Consider two distinct alliances $\Ac_1$ and $\Ac_2$ of laminations. Since they are distinct, we may assume that
$\Ac_1$ is generated by a non-central minimal lamination $\lam_1$ and consists of all laminations that tune $\lam_1$.
Suppose that $\hlam$ is a lamination that belongs to both $\Ac_1$ and $\Ac_2$. Choose a minimal sublamination $\lam_2$
of $\hlam$. Then there is a critical portrait compatible with both $\lam_1$ and $\lam_2$. By Lemma \ref{l:crplam},
this implies that $\lam_1=\lam_2$ and, hence, $\Ac_1=\Ac_2$.

Now we prove that alliances of non-central laminations are the same as and their T-classes.
Indeed, let $\lam$ be a minimal non-central lamination. Its alliance $\Ac(\lam)$ consists of all laminations $\hlam$ such that $\lam$ is
a sublamination of $\hlam$. Hence $\Ac(\lam)$ is contained in the T-class of $\lam$. On the other hand, if two laminations
are tuning related then they have a common critical portrait compatible with them both. Hence in any chain of tuning related
laminations the consecutive sets of compatible critical portraits have nonempty intersections.
By Lemma \ref{l:crplam}, it follows that
any lamination from the T-class of $\lam$ comes from $\Ac(\lam)$.  This proves the last claim of the theorem.
\end{proof}

Recall (see Definition \ref{d:spalam}) that $\ssl_3$ is
the space of all invariant cubic laminations.
By Theorem \ref{t:no-share-1}, the following sets are well defined.

\begin{dfn}\label{d:notalfi}
For $\lam\in \ssl_3$, let $\A(\lam)$ be the alliance of laminations to which
$\lam$ belongs; for a critical portrait $\K$, let $\F(\K)$ be the fiber of critical portraits to which
$\K$ belongs.
\end{dfn}

%Theorem \ref{t:alli} is the key theorem of this section.
Recall \cite{Dav86} that an \emph{upper semicontinuous (USC) partition} of a
 compact metrizable space $S$ is defined as a partition $\Zc$ of $S$ into
 compact subsets with the property that, for every $Z\in\Zc$ and every open $U\supset Z$,
 there is an open neighborhood $V\subset U$ of $Z$ such that $Z'\in\Zc$ lies in $U$
 whenever $Z'\cap V\ne\0$.
Equivalently, if a sequence $z_n\in Z_n$ converges to $z\in Z$, then every convergent
 subsequence of every sequence $z'_n\in Z_n$ has its limit in $Z$
 (the equivalence is established in Proposition 2 and Exercise 11 of \cite[Section 1]{Dav86}).
Also, USC partitions are basically the same as closed equivalence relations:
 an equivalence relation $R$ on $S$ is closed as a subset of $S\times S$ if and only if
 the collection of all $R$-classes is an USC partition of $S$
 (this is an immediate consequence of the above characterization in terms of convergent
 (sub)sequences).

\begin{thm}\label{t:alli}
Alliances of laminations form a USC partition of
$\ssl_3$. Fi\-bers of critical portraits form a USC-partition $X_3$ of $\crp$.
The map $\Psi:\ssl_3\to X_3$ that associates to each lamination $\lam$ the fiber of
critical portraits compatible with minimal lamination(s) from $\A(\lam)$ is well defined and continuous.
The union of non-central fibers of critical portraits is open and dense in $\crp$.
\end{thm}

Recall (Definitions \ref{d:spalam} and \ref{d:cripo}) that both $\ssl_3$ and $\crp$ are compact metrizable spaces.

\begin{proof} By Theorem \ref{t:no-share-1}, alliances
of laminations partition $\ssl_3$ and fibers of critical portraits partition $\crp$.
Clearly, a non-central alliance of laminations is closed in $\ssl_3$, and
a non-central fiber of critical portraits $\crp(\lam)$, with $\lam$ a non-central minimal lamination, is closed in $\crp$.

We claim that the central alliance of laminations is closed. Let $\lam_i\to \lam$ where
$\lam_i$ are central laminations with central minimal laminations $\lam'_i$; by Theorem \ref{t:laclo} we may
assume that $\lam'_i$ converge to an invariant lamination $\lam'\subset \lam$. Choose  flower-like sets $F'_i$ compatible
with $\lam'_i$ for every $i$. By Lemma \ref{l:flowerc} we may assume that $F'_i\to F'$ with $F'$
containing a flower-like set %$F'$
compatible with $\lam'$.
Hence minimal laminations of $\lam'$ are central as desired. This implies that the central alliance of critical portraits is closed.

To show that alliances of laminations form a USC-partition of $\ssl_3$,
let $\lam_i\to \lam$ and $\lam'_i\to \lam'$ be two
%converging
sequences of laminations where $\lam_i$ and $\lam'_i$ belong to the same non-central alliance of laminations
with a minimal lamination $\lam''_i$ for every $i$. Assume that $\lam''_i\to \lam''$; then
$\lam''\subset \lam\cap \lam'$ and, hence,
 that $\lam$ and $\lam'$ have a common minimal lamination,
and belong to the same alliance
(non-central or central).

To show that fibers of critical portraits form a USC-partition of $\crp$,
let $\K_i\to \K$ and $\K'_i\to \K'$ be two sequences of critical portraits,
where $\K_i$ and $\K'_i$ belong to the same non-central fiber of critical portraits
with minimal laminations $\lam''_i$ compatible with $\K_i$ and $\K'_i$ for every $i$.
Assume that $\lam''_i\to \lam''$. Then %It follows
%that
$\lam''$ is compatible with both $\K$ and $\K'$, and
%Hence
$\K$ and $\K'$ belong to the same fiber of critical portraits (non-central or central).
It is easy to see that
%The same arguments show that
the map $\Psi:\ssl_3\to X_3$ is well defined and continuous.

We claim that the union $\U$ of non-central fibers of critical portraits is open and dense in $\crp$.
The set $\U$ is open since its complement is the central fiber of critical portraits which is closed.
Let $\K=\{\oc, \oy\}$ be a critical portrait such that the orbits of $\si_3(\oc)$ and $\si_3(\oy)$ are
dense in $\uc$. We claim that $\K$ is non-central. Indeed, if it is central, then there is a central minimal lamination $\lam$
compatible with $\K$. Let $C\supset \oc$ and $Y\supset \oy$ be the critical sets of $\lam$.
If $C$ is infinite, then $C$ is a (pre)periodic gap.
This follows from Theorem \ref{t:kiwinf}.  % Theorem 1.1 of \cite{kiw02}, which is stated for polynomials but uses
 %only combinatorial arguments, hence it is true also for invariant laminations.
Now, we obtained a contradiction with the density of $\si_3(\oc)$.
Thus, $C$ (and $Y$) are finite. On the other hand,
$\lam$ is compatible with a flower-like set $F$. In particular, there is a cycle of infinite gaps compatible with $\lam$.
Let $G$ be a gap from this cycle that contains a critical chord.
Since endpoints of $\oc$
and $\oy$ cannot be vertices of $G$, we may assume that $\oc=\ol{xy}$ where $x, y$ belong to distinct components $I, J$ of $\uc\sm G$.
Since $\lam$ and $F$ are compatible, the finite concatenation of edges of $C$ that connects the endpoints of $\oc$
must pass through an endpoint of, say, $I$. As $\si_3(\oc)$ visits $I$ infinitely often, each time the corresponding
image of a vertex of $C$ coincides with endpoint of $I$. The fact that $C$ has finitely many vertices implies now that a vertex of $C$ is preperiodic.
Together with the density of the orbit of $\si_3(\oc)$ this yields a contradiction.
\end{proof}

\section{The model}\label{s:model}

The connection between abstractly defined laminational equivalence relations
and the laminational equivalence relations generated by polynomials is established in
the following fundamental result of Kiwi.

\begin{thm}\cite[Theorem 1]{kiwi97}
\label{t:kiwi1}
Let $Q$ be a polynomial of degree $d$ without Cremer or Siegel cycles such that $J_Q$ is connected.
Then all $\sim_Q$-classes are finite, and there exists a monotone map $p:J_Q\to \uc/\sim_Q$ that semiconjugates
$Q|_{J_Q}$ with the induced map $f_{\sim_Q}:\uc/\sim_Q\to \uc/\sim_Q$; the map $p$ is one-to-one on all (pre)periodic points of $Q$ in $J_Q$.
Moreover, if $\sim$ is a laminational equivalence such that $f_\sim:\uc/\sim\to \uc/\sim$ has no Siegel gaps, then there exists
a polynomial $Q$ such that $\sim=\sim_Q$.
\end{thm}

The next theorem complements Theorem \ref{t:kiwi1}.

\begin{thm}\cite[Theorem 2 and Lemma 37]{bco13}
\label{t:bco1}
Suppose that $Q$ is a polynomial of degree $d$ such that $J_Q$ is connected.
There exists a monotone map $p:J_Q\to \uc/\sim_Q$ that semiconjugates
$Q|_{J_Q}$ with the induced map $f_{\sim_Q}:\uc/\sim_Q\to \uc/\sim_Q$.
If $Q$ has a parabolic or attracting periodic Fatou domain, then its boundary is not one $\sim_Q$-class.
If $\mathbf{h}$ is a finite periodic $\sim_Q$-class,
then the impressions of all angles from $\mathbf{h}$ coincide and equal a periodic repelling or parabolic point of $Q$.
\end{thm}

Theorem \ref{t:bco1} is weaker than Theorem \ref{t:kiwi1} as it does not claim the finiteness of $\sim_Q$-classes.
Still, it can be helpful as it applies to all polynomials with connected Julia sets. By Theorem \ref{t:bco1},
each $\sim_P$-class $\mathbf{h}$ corresponds to the (connected and closed)
union of impressions of rays whose arguments are elements of $\mathbf{h}$.
We will call this union the \emph{impression} of $\mathbf{h}$.

Now we can move on to describing our approach.
A point $x$ is \emph{(pre)\-re\-pel\-ling} if it eventually maps to a \emph{repelling} periodic point.
An unordered pair of rational angles $\{\al,\be\}\subset\Q/\Z$ is \emph{(pre)repelling} if the
external rays with arguments $\al$ and $\be$ land at the same (pre)repelling point.
Let $\rep_P$ be the set of all (pre)repelling pairs of angles. Observe that this
set (and related to it concepts defined later) can be considered
for a polynomial $P$ of any degree if its Julia set is connected.

\begin{dfn}\label{d:replam}
Let $\bw_P$ be the equivalence relation on $\rb/\Z$ given by $\al\bw_P \be$ if
$\{\al,\be\}\in\rep_P$ or $\al=\be$.
Let $\lam^{rep}_P$ be the set of all edges of the convex hulls in $\cdisk$ of all $\bw_P$-classes
\emph{and the limits of these edges}.
By Theorem \ref{t:laclo}, the set $\lam^{rep}_P$ is a lamination (cf. \cite{bmov13}).
\end{dfn}

%The equivalence relation
By definition $\bw_P$ deals with \emph{repelling} cycles but \emph{not} with parabolic cycles.
Yet, %However
$\lam^{rep}_P$ may have periodic gap-leaves
associated with parabolic points of $P$ because limits of edges of $\bw_P$-classes
are also leaves of $\lam^{rep}_P$.

The lamination $\lam^{rep}_P$ is associated with an equivalence relation $\sim_{\lam^{rep}_P}$ on $\uc$
so that all gap-leaves of $\lam^{rep}_P$ are convex hulls of $\sim_{\lam^{rep}_P}$-classes.
Since $\sim_{\lam^{rep}_P}$ is the closure of $\bw_P$, in what follows we simply denote it by $\bw_P$.
Clearly, if $\al \bw_P \be$ then $\al \sim_P \be$. Therefore $\sim_P$ tunes $\bw_P$ (in other words,
$\sim_P$ may add more connections among arguments of external rays compared to $\bw_P$).
Thus, if a polynomial $g$ is such that
$\sim_g=\bw_P$, then $P$ tunes $g$; therefore, if $g$ belongs to a T-class then so does $P$.
We use this observation all the time without additional
explanations because it allows us to replace, in a lot of arguments, $P$ and $\sim_P$ by $g$ such that $\sim_g=\bw_P$.

\begin{dfn}
\label{d:poly-fiber}
For $P\in \Cc_3$, the \emph{fiber of critical portraits $\F_P$ generated by $P$} is defined as follows.

\begin{enumerate}

\item If $\lam^{rep}_P$ is non-central, then, by Theorem \ref{t:no-share-1}, it has a unique non-central minimal lamination
denoted by $\lam^{min}_P$; denote by $\F_P$ the fiber $\crp(\lam^{min}_P)$ of critical portraits compatible with $\lam^{min}_P$.

\item If $\lam^{rep}_P$ is central (e.g., if $\lam^{rep}_P$ is trivial), we let $\F_P=\F_0$ be the central fiber of critical portraits
(e.g., $\F_0$ serves all polynomials $P$ with empty $\lam^{rep}_P$).

\end{enumerate}

\end{dfn}

By Theorem \ref{t:alli}, the sets $\F_P$ and $\F_0$ are closed.

\begin{lem}\label{l:reg-all}
For $P\in\Cc_3$ with non-central fiber $\F_P$ of critical portraits, $\F_P=\crp(\lam^{rep}_Q)$ for a
polynomial $Q$ (possibly $Q\ne P$); here $\lam^{rep}_Q=\lam^{min}_P$.
\end{lem}

\begin{proof}
By Lemma \ref{l:periodense}, the lamination $\lam^{min}_P$ is perfect.
Then $\lam^{min}_P$ has no infinite periodic gaps of degree 1 by Lemma \ref{l:crit-must}.
By Theorem \ref{t:kiwi1}, there exists a polynomial $Q$ such that $\sim_Q=\sim_{\lam^{min}_P}$
and $\lam_{\sim_Q}=\lam^{min}_P$.
By Lemma \ref{l:periodense}, the lamination $\lam^{min}_P$ has infinitely many periodic gap-leaves,
 and, for any such gap-leaf, the iterated pullbacks of its leaves are dense in $\lam^{min}_P$.
We can choose a periodic gap-leaf like that so that, in terms of $Q$, it is associated with
a repelling periodic point. It follows that $\lam^{rep}_Q=\lam_{\sim_Q}=\lam^{min}_P$.
Then, by definition, $\F_P=\crp(\lam^{min}_P)=\crp(\lam^{rep}_Q)$, as desired.
\end{proof}

A point $x$ is \emph{($P$-)stable}
if its forward orbit is finite and contains no critical points and no
non-repelling periodic points. Lemma \ref{l:rep} follows from \cite{hubbdoua85}
(cf. \cite[Lemma B.1]{gm93}).

\begin{lem}\cite{hubbdoua85, gm93}
 \label{l:rep}
Let $g\in \Cc_3$ be a polynomial, and $z$ be a $g$-stable point. If an
external ray $R_g(\theta)$ with rational argument $\theta$ lands at
$z$, then, for every polynomial $\tilde g$ sufficiently close to $g$,
the external ray $R_{\tilde g}(\theta)$ lands at a $\tilde g$-stable point $\tilde z$ close
to $z$. Moreover, $\tilde z$ depends holomorphically on $\tilde g$.
\end{lem}

Since $\crp$ is a compact metrizable space, we can talk about
convergence of compact subsets of $\crp$ into each other.
If $\{\A_i\}$, $i=1$, $2$, $\dots$ is a sequence of compact subsets of $\crp$ that converge into
a compact subset $Y$ of $\crp$, we write $\A_i\hto Y$.

%a in the plane,
%and that %$\A$
%$\{\B_\al\}$ is a family of compacta.
%We write $\A_i \hto \A$ if the limit of every convergent sequence $C_i\in\A_i$ in the Hausdorff metric is a compactum from $\{\B_\al\}$.
%Note that $\A_i\hto \A$ implies $\A_i\hto \A'$ for every $\A'\supset\A$.

\begin{lem}\label{l:regulim}
Consider a sequence of polynomials $P_i\in\Cc_3$ converging to a polynomial $P\in\Cc_3$.
If $\lam^{rep}_P$ is non-central and $\lam^{rep}_{P_i}$ converge to some lamination $\lam'$,
then $\lam^{rep}_P$ and $\lam'$ belong to the same non-central
alliance of laminations, $\lam^{rep}_{P_i}$ are non-central
for all sufficiently large $i$, and $\F_{P_i}\hto \F_P$.
\end{lem}

\begin{proof}
By Lemma \ref{l:periodense}, the lamination $\lam^{min}_P$ is perfect, has infinitely many periodic
gap-leaves, and all iterated pullbacks of any periodic leaf of $\lam^{min}_P$ are
dense in $\lam^{min}_P$. Choose a repelling $P$-periodic point $x$ that is not
an eventual image of a critical point of $P$. Take any iterated $P$-preimage $y$ of
$x$ and write $A_y(P)$ for the set of arguments of all $P$-external rays landing at $y$.
By Lemma \ref{l:rep}, there exists $N_y$ such that for every $i>N_y$,
there is a point $y_i$ with $A_{y_i}(P_i)=A_y(P)$.
It follows that $\lam^{min}_P\subset \lam'$. This implies the first
claim of the lemma which, by Theorem \ref{t:alli}, implies the other claims.
\end{proof}

Recall: $X_3$ is the partition of $\crp$ formed by the fibers of critical portraits.

\begin{dfn}\label{d:pi}
Define the map $\eta_3:\Cc_3\to X_3$ by the formula $\eta_3(P)=\F_P$.
\end{dfn}

We may talk of central or non-central polynomials, as well of central or non-central fibers of polynomials,
 as is specified in the following definition.

\begin{dfn}\label{d:pifibers}
A polynomial $P$ is a \emph{non-central polynomial} if $\lam^{rep}_P$ is non-central. A polynomial $P$ is \emph{central}
if $\lam^{rep}_P$ is central. A minimal non-central lamination $\lam^{min}$ defines the corresponding
\emph{non-central fiber of polynomials} which consists of polynomials $f$ such that $\lam^{rep}_f\supset \lam^{min}$.
The \emph{central fiber of polynomials} is the family of polynomials $g$ such that $\lam^{rep}_g$ is trivial or contains a minimal lamination
$\lam^{min}_g$ compatible with a flower-like set.
\end{dfn}

\begin{thm}
\label{t:main}
The map $\eta_3$ is continuous.
\end{thm}

\begin{proof}
Consider a sequence $P_i\to P$ of polynomials, and set $\F_i=\F_{P_i}$.
If $\F_i\not\hto \F_P$, then by Theorem \ref{t:alli} we may assume that
$\F_i\hto \F'\ne \F_P$ where (1) the limit $\F'$ is a fiber of critical portraits, and (2)
all fibers of critical portraits $\F_i$ are non-central, or $\F_i=\F_0$ for every $i$.
By Lemma \ref{l:regulim}, this situation is impossible if $\F_P$ is non-central, or
if $\F_i=\F_P=\F_0$. It remains to assume that all fibers of critical portraits
$\F_i$ are non-central, $\F'$ is non-central, $\F_P=\F_0$ is central, and bring this to a contradiction.
Let $\lam'$ be the non-central minimal lamination such that $\F'=\crp(\lam')$.

If $\lam^{rep}_{P_i}\to \lam$ for a non-central lamination $\lam$,
then $\lam'\subset \lam$ by Theorem \ref{t:alli}. By Lemma \ref{l:periodense}, there are \emph{infinitely} many periodic
gap-leaves of $\lam'$; since $\lam'\subset \lam$, it follows from Lemma \ref{l:limleaf} that any such
gap-leaf is a gap-leaf of $\lam$. Let $B_P$ be the set of vertices of
gap-leaves of $\lam^{rep}_P$ associated with parabolic points. Choose a periodic
gap-leaf $G$ of $\lam'$ so that no vertex of $G$ belongs to $B_P$ (this is always possible as
$B_P$ is \emph{finite}). By Lemma \ref{l:limleaf}, the set $G$ is a gap-leaf of $\lam_i$ for sufficiently large $i$.
By Theorem \ref{t:kiwi1} and by our assumptions, $G$ is associated with a repelling periodic point, say, $y_i$ of $P_i$.

Evidently, $G$ is a gap-leaf of $\lam^{rep}_P$, too. Indeed, $P$-external rays corresponding to the vertices of $G$
land at repelling points by the choice of $G$. By the above, $P_i$-external rays with the same arguments all land at $y_i$.
If the $P$-external rays mentioned above do not land at the same point, then
by continuity (Lemma \ref{l:rep}) neither do the corresponding $P_i$-external rays, a contradiction. So, $G$ is a gap-leaf
of $\lam^{rep}_P$, too. Let $y$ be the repelling periodic point of $P$  associated with $G$.
By the choice of $G$, no critical point of $P$ ever maps to $y$, hence, by Lemma \ref{l:rep},
all pullbacks of $G$ in $\lam^{rep}_P$ eventually become
gap-leaves of $\lam_{i}$, and, therefore, of $\lam'$.
By the properties of non-central minimal laminations listed in Lemma \ref{l:periodense}, it follows that
$\lam'\subset \lam^{rep}_P$. This contradicts the assumption that $\F_P=\F_0$ and completes the proof.
\end{proof}

\section{Connectedness of the fibers of polynomials}\label{s:conn-fib}

Let $P\in \Cc_3$. Recall \cite{kiw01} that $\lambda(P)$ is an
equivalence relation on $\Q/\Z$ such that $(\al, \be)\in \la(P)$
if and only if the external rays $R_P(\alpha)$, $R_P(\beta)$ land at the same point.
Let $\ol{\la(P)}$ be the closure of the equivalence relation $\la(P)$.

%two arguments $\alpha$,
%$\beta\in\Q/\Z$ are equivalent if and only if the external rays
%$R_P(\alpha)$, $R_P(\beta)$ land at the same point.
%Write $\ol{\la(P)}$ for the closure of the equivalence relation $\la(P)$.

\begin{lem}\cite[Lemma 3.9]{kiw01}
\label{l:3.9}
For a polynomial $P$, the relation $\la(P)$ is closed as a subset of $\Qb\times \Qb$.
If $(x, y)\in \ol{\la(P)}$ is (pre)periodic and $x\ne y$
then $(x, y)\in \la(P)$ and the external rays $R_P(x)$ and $R_P(y)$ form a cut.
\end{lem}

For $f_0\in\Cc_d$, the \emph{combinatorial renormalization domain}
$\Cc(f_0)$ is defined \cite{IK12} as $\{f\in\Cc_d\mid \lambda(f)\supset
\lambda(f_0)\}$ (i.e., if two angles $\al$ and $\be$ are $\la(f_0)$-equivalent, then they are $\la(f)$-equivalent).

\begin{lem}\label{l:lamvsCc}
Assume that $P$, $P_0\in\Cc_3$, every nondegenerate leaf of $\lam^{rep}_{P_0}$ is in a
gap-leaf of $\lam^{rep}_P$, and $P_0$ has no neutral cycles.
Then $P\in\Cc(P_0)$.
\end{lem}

Note that Lemma \ref{l:lamvsCc} is applicable in the case $\lam^{rep}_P\supset\lam^{rep}_{P_0}$.

\begin{proof}
Since $P_0$ has no neutral cycles, then $\la(P_0)=\sim_{P_0}$ on rational angles. If now $\al$ is $\la(P_0)$-equivalent
to $\be$, then $\al\sim_{P_0} \be$. By the assumptions,
if $\al \sim_{P_0} \be$ then $\al\sim_P \be$; by Lemma \ref{l:3.9}, the angle $\al$ is $\la(P)$-equivalent
to $\be$. Thus, $\la(P_0)\subset \la(P)$ as desired.
\end{proof}

Theorem \ref{t:SW} easily follows from much stronger results of Shen--Wang \cite{shen2020primitive},
Wang \cite{wang2021primitive}, and Kozlovski--van Strien \cite{KozlovskivanStrien2009}.

\begin{thm}\label{t:SW}
Suppose that $\lam^{rep}_P$ is perfect for some $P\in\Cc_3$. If there exists no
infinite sequence $P_n\in\Cc_3$ such that
$\Cc(P_1)\supsetneq\Cc(P_2)\supsetneq\dots$ and $\bigcap_n \Cc(P_n)\supset \Cc(P)$
then the set $\Cc(P)$ is connected.
In particular, $\Cc(P)$ is connected provided that $\lam^{rep}_P$ is a perfect minimal lamination.
\end{thm}

\begin{proof}
By Theorem \ref{t:kiwi1}, there is $P_0\in\Cc_3$ without neutral cycles such that $\lam^{rep}_P=\lam^{rep}_{P_0}$.
By Lemma \ref{l:lamvsCc}, $\Cc(P)=\Cc(P_0)$.
So, we may assume that $P$ has no neutral cycles.
If all periodic points of $P$ are repelling, then $\Cc(P)$ is a singleton, by \cite[Theorems 1.1 and 1.2]{KozlovskivanStrien2009}
and the assumptions. Assume that $P$ has (super)attracting domains.
If $P$ is hyperbolic, the result follows from the main result of Shen and Wang \cite{shen2020primitive}.
Finally, if one critical point $c_1$ of $P$ lies in a (super)attracting periodic basin,
the other critical point $c_2$ is in the Julia set, and $P$ has the properties from
the statement of the theorem, then $\Cc(P)$ is also connected by \cite[Theorem A]{wang2021primitive}.
Note: only the nonrenormalizable case is considered in \cite{wang2021primitive} but the
 extension to the finitely renormalizable case is straightforward; the connectedness of
 $\Cc(P)$ follows from the bijectivity of the straightening map in the same way as in \cite{shen2020primitive}.
\end{proof}

Let $\F$ be a non-central fiber of critical portraits, and
consider the corresponding non-central fiber $\eta_3^{-1}(\F)$ of polynomials. Lemma \ref{l:regfibren}
completes the proof of the fact that non-central fibers of polynomials are connected. To avoid cumbersome
notation, set $\lam_{\ol{\la(P)}}=\lam^{rat}_P$.

\begin{lem}
  \label{l:regfibren}
If $\F$ is a non-central fiber of critical portraits, then $\eta_3^{-1}(\F)=\Cc(P)$
for some $P\in \Cc_3$ without neutral cycles but with infinitely many
periodic cuts. Any non-central fiber of polynomials is connected, and if $\phi:\Cc_3\to Y$
 is T-stable, then the $\phi$-image of a non-central fiber is a point.
\end{lem}

\begin{proof}
By Lemma \ref{l:reg-all}, $\F=\crp(\lam^{rep}_P)$
for $P\in\Cc_3$ without neutral cycles such that $\lam^{rep}_P$ is a perfect minimal lamination.
By definition, in this case $\lam^{rep}_P=\lam^{rat}_P$.
By Lemma \ref{l:3.9}, every periodic
gap-leaf of $\lam^{rep}_P$ is associated with a periodic cut in $\la(P)$.
Since, by Lemma \ref{l:periodense}, the lamination $\lam^{rep}_P$ has infinitely many periodic finite gap-leaves,
 $P$ has infinitely many periodic cuts.
As $\eta_3^{-1}(\F)=\Cc(P)$ by definition, a non-central fiber of polynomials is connected,  by Theorem \ref{t:SW}.
The last claim of the lemma is immediate.
\end{proof}

Observe that for $P\in\Cc_3$ such that $\Cc(P)$ is
 a non-central fiber of polynomials, the periodic Fatou domains of $P$ have pairwise disjoint closures.

%\section{Laminations compatible with flower-like sets}\label{s:maint}
%Let us sketch the plan of the proof of the main theorem.

%\section{Polynomials with non-repelling fixed points and laminations with flower-like sets}

\section{Main Theorem}

We have constructed
the map $\eta_3:\Cc_3\to X_3$ by the formula $\eta_3(P)=\F_P$. This models $\Cc_3$ %the cubic
%connectedness locus
by the quotient space $X_3$ of the space of all cubic critical portraits constructed
without invoking polynomials so that the modeling space is combinatorial.
In the rest of the paper, we prove that
$\eta_3$ solves the Main Problem. Corollary \ref{c:non-central} follows from
Theorem \ref{t:no-share-1}.

\begin{cor}\label{c:non-central}
A non-central fiber of polynomials is a T-class.
\end{cor}

The central fiber of polynomials $\eta_3^{-1}(\Fc_0)$ is the set of
polynomials $f$ such that $\lam^{rep}_f$ is either empty or central (i.e., it contains a minimal lamination $\lam^{min}_f$
compatible with a flower-like set). This is done in Definition \ref{d:poly-fiber} (recall that
$\lam^{rep}_f$ is introduced in Definition \ref{d:replam} and is, loosely, based upon the family of repelling cuts of $f$).
To complete the proof of the Main Theorem, we need to show that
it equals the union of T-classes of all polynomials with a non-repelling fixed point
which equals the union of T-classes of all polynomials with a neutral fixed point. To this end we first establish
the connection between $\lam^{rep}_f$ and $\lam_f$ in the context of the central fiber of polynomials
(recall that $\lam_f$ is introduced in Definition \ref{d:fullami} and is based, loosely, upon the impressions of external rays
to the Julia set of $f$ and the way they intersect each other).
%Recall also that we consider cubic polynomials $f,$ $g$, $h$, $P$ etc with connected Julia sets.

\begin{comment}

In the rest of the paper, we
show that it coincides with the union of T-classes of all polynomials with non-repelling fixed points.
This will complete, with the help of some additional arguments, the proof of the Main Theorem.
Recall that $\lam^{rep}_f$ is constructed based on the cuts with vertices at repelling periodic points and their
preimages (Definition \ref{d:replam}) while $\lam_f$ is based upon (non-)intersection of impressions of rays
with various arguments (Definition \ref{d:fullami}). Recall that we consider cubic polynomials
(normally denoted $f,$ $g$, $h$, $P$ etc) with connected Julia sets.

\end{comment}

\begin{lem}\label{l:central2}
A polynomial $h\in\Cc_3$ belongs to the
central fiber of polynomials $\eta_3^{-1}(\Fc_0)$ if and only if $\lam_h$ is central.
\end{lem}

Recall that the empty lamination is central (see Definition \ref{d:alliance}).

\begin{proof}
%We claim that the lamination $\lam^{rep}_h$ is central if and only if $\lam_h$ is central.
Let us begin by proving that if $\lam_h$ is non-central then $\lam^{rep}_h$ is non-central.
Assume that $\lam_h$ is non-central. Then it contains a unique non-central minimal lamination $\hlam_h\subset \lam_h$.
By Lemma \ref{l:periodense}, $\hlam_h$ has infinitely many periodic gap-leaves and the iterated pullbacks of any periodic leaf
of $\hlam_h$ is dense in $\hlam_h$. Hence, by Theorem \ref{t:bco1} there exists a periodic gap-leaf $G$ of $\hlam_h$ which
is associated with a repelling cut of $h$ and is, therefore, a periodic gap-leaf of $\lam^{rep}_h$. It follows that $\hlam_h\subset \lam^{rep}_h$.
By Theorem \ref{t:no-share-1}, non-central alliances of laminations are disjoint from the central alliance of laminations.
The inclusion $\hlam_h\subset \lam^{rep}_h$ and the fact that $\hlam_h$ is non-central imply that $\lam^{rep}_h$ is non-central.

Assume now that $\lam^{rep}_h$ is non-central. We need to show that then $\lam_h$ is non-central.
%To begin, let us show that $\lam_h$ is non-empty.
Choose a unique minimal sublamination $\lam^{min}_h$ of $\lam^{rep}_h$.
By Lemma \ref{l:periodense}, we can choose a periodic gap-leaf $G$ in $\lam^{min}_h$ with dense pullbacks
in $\lam^{min}_h$. The gap-leaf $G$ and its pullbacks in $\lam^{min}_h$ separate all gap-leaves of $\lam^{min}_h$
from each other. Therefore, the cuts on the complex plane associated with $G$ and its pullbacks in $\lam^{min}_h$
separate the impressions of all external rays associated with all gap-leaves of $\lam^{min}_h$ from each other.
This implies that the only possible intersections among impressions of external rays of $h$ are possible
between external rays with arguments that belong to a (possibly infinite) gap or a leaf of $\lam^{min}_h$, otherwise
(i.e., if two angles belong to distinct gaps or leaves of $\lam^{min}_h$) their impressions are separated by infinitely many
(pre-)repelling cuts. By definition of $\lam_h$ (Definition \ref{d:fullami}) this implies that $\lam_h\supset \lam^{rep}_h$,
which by Theorem \ref{t:no-share-1} implies that $\lam_h$ is non-central.
\end{proof}

The next lemma is immediate and is left to the reader.

\begin{lem}\label{l:rep-triv}
If $\lam^{rep}_f$ is empty, then $f$ has a non-repelling fixed point and, therefore, belongs to the
union of T-classes of polynomials with a non-repelling fixed point.
If $\lam_f$ is empty, then $f$ belongs to the
union of T-classes of polynomials with a non-repelling fixed point.
\end{lem}

The next lemma complements Lemma \ref{l:rep-triv}.

\begin{lem}\label{l:rep-nontriv}
Suppose that both $\lam^{rep}_f$ and $\lam_f$ are nonempty, and that $\lam_f$ is central.
Then $f$ belongs to the T-class of some polynomial $P$ such that either $P$ has a non-repelling fixed point
or $\lam_P$ is minimal, has no Siegel or caterpillar gaps, but has a flower-like set.
\end{lem}

\begin{proof}
If $\lam_f$ has an invariant Siegel gap, it follows from Theorem \ref{t:bco1}
 that $f$ has a non-repelling fixed point, which implies the desired (we can set $P=f$).
From now on, assume that $\lam_f$ does not have an invariant Siegel gap.

For each of the periodic Siegel gaps of $\lam_f$ (if any),
 remove its critical edge(s) and all iterated $\si_3$-pullbacks of it.
Since all edges of Siegel gaps are isolated, the set $\hlam_f$ of remaining leaves is closed.
Moreover, $\hlam_f$ is an invariant lamination, by Definition \ref{d:sibli}.
Observe that, since $\lam_f$ does not have an invariant Siegel gap by our assumption,
the above described removal of some (countably many, to be precise) leaves from $\lam_f$ does not create
the empty lamination (alternatively, in establishing this one can use the fact that $\lam^{rep}_f$ is nonempty).
Thus, $\hlam_f$ is nonempty and, by construction, has no Siegel gaps.

Choose a minimal sublamination $\hlam^{min}_f$ of $\hlam_f$ compatible with a flower-like set $F$
 (this is possible because $\lam_f$ is central).
By construction, $\hlam^{min}_f$ has no Siegel gaps. Moreover, $\lam_f$ has no
 caterpillar gaps because $\lam_f$ is a q-lamination, and q-laminations cannot have caterpillar gaps as q-laminations cannot
have a critical edge with periodic and non-periodic endpoints. Thus, $\hlam^{min}_f$ has no caterpillar gaps either.
By Theorem \ref{t:kiwi1}, there exists a polynomial $h$ such that $\lam_h=\hlam^{min}_f$.
By construction, $f$ belongs to the T-class of $h$.
Consider cases.

First suppose that $\lam_h=\hlam^{min}_f$ is countable. Then by Lemma \ref{l:count0} it \emph{has}
a flower-like set which does not involve Siegel gaps or caterpillar gaps. Hence we can set
$P=h$ and complete the proof in this particular case.

%If $\lam_h=\hlam^{min}_f$ has a quadratic invariant gap, then $P$
% has a non-repelling fixed point; otherwise, it has a flower-like set, which consists of a
% gap-leaf and Fatou non-caterpillar gaps rotating around it.
%This proves the lemma %in the case
%when $\lam_h=\hlam^{min}_f$ is countable (we can set $P=h$).

Consider the case when $\lam_h=\hlam^{min}_f$ is perfect (minimal laminations are either countable of perfect).
Then, by Lemma \ref{l:flower-prop}, there exists
 a minimal proper lamination $\lam'$ compatible with $\lam_h$ that has a flower-like set $E'$ and has no Siegel or caterpillar gaps.
By Theorem \ref{t:kiwi1}, there exists a polynomial $P$ such that $\lam_P=\lam'$. Next
consider the union $\lam_h\cup \lam'=\tlam$.
The lamination $\tlam$ is proper and has no Siegel gaps, by Lemma \ref{l:unite}.
By Theorem \ref{t:kiwi1}, there exists a polynomial $g$ such that $\lam_g=\tlam$.
It follows that $h$, $P$, and $g$ belong to the same T-class of polynomials.
Since $P$ is exactly of the type mentioned in the lemma, this completes the proof.
\end{proof}

%Recall that the central fiber of polynomials $\eta_3^{-1}(\Fc_0)$ is the set of
%polynomials $f\in \Cc_3$ such that $\lam^{rep}_f$ is either empty or
%central (i.e., it contains a minimal lamination $\lam^{min}_f$
%compatible with a flower-like set).

Theorem \ref{t:central-f} describes the central fiber of polynomials.

\begin{thm}\label{t:central-f}
The central fiber of polynomials is the union of T-classes of polynomials with a non-repelling fixed point
which coincides with the union of T-classes of polynomials with a neutral fixed point.
\end{thm}

%Lemma \ref{l:1way1} is based on Theorem \ref{t:bco1} and some additional arguments.

Recall that a polynomial $g\in \Cc_3$ belongs to the central fiber of polynomials if and only if
$\lam_g$ is central.

\begin{lem}\label{l:1way1}
If $P\in \Cc_3$ has a non-repelling fixed point, then either $\lam_P$ is empty or it has a flower-like set
so that $P$ belongs to the central fiber of polynomials.
If $f\in \Cc_3$ belongs to the T-class of $P$ then $f$ also belongs to the central fiber of polynomials. Thus,
the union of T-classes of polynomials with a non-repelling fixed point is contained in the central fiber of polynomials.
\end{lem}

\begin{proof}
By Theorem \ref{t:bco1}, we may assume that $P$ has a Cremer or Siegel fixed point $a$; assume also
that $\lam_P$ is not empty.
By Theorem \ref{t:bco1}, each $\sim_P$-class $\mathbf{h}$ corresponds to the (connected and closed)
union of impressions of rays with arguments from $\mathbf{h}$.
We call this union the \emph{impression} of $\mathbf{h}$. %, cf . Section \ref{ss:main-prob}.
Suppose that $a$ is a Cremer fixed point.
Consider an angle $\theta$ such that $a\in I_P(\theta)$ (recall: $I_P(\theta)$ is the impression of angle $\theta$
in the dynamical plane of $P$) and then
all the angles from the $\sim_P$-class $\mathbf{h}$ of $\theta$. Since $a$ is fixed,
$\si_3(\theta)\in \mathbf{h}$, which implies that $\mathbf{h}$ is a $\si_3$-invariant $\sim_P$-class.
By Theorem \ref{t:bco1}, the convex hull of $\mathbf{h}$ is an infinite invariant gap, i.e. a flower-like set.

Assume that $a$ is a Siegel point. Let $B_a$ be its invariant Siegel domain. Suppose that there is an angle $\al$
such that $I_P(\al)\supset \bd(B_a)$ and consider the $\sim_P$-class $\mathbf{h}$ of $\al$. As above, it follows that $\mathbf{h}$
is an infinite invariant $\sim_P$-class, and its convex hull is the desired flower-like set of $\lam_P$.

Now, sup\-pose that no impression of an external ray contains $\bd(B_a)$
and $\lam_P$ has no flower-like sets. Choose
a finite invariant gap-leaf $G$ of $\lam_P$; by Theorem \ref{t:bco1} the corresponding external rays land on the same point, say, $a$,
and have impressions coinciding with $\{a\}$.
Take the closure of all rational gap-leaves of $\lam_P$ to obtain a lamination $\hlam$.
Let $\Bc$ be the family of all $\sim_P$-classes with impressions
that intersect $\bd(B_a)$; then $\bigcup \Bc$ is a closed infinite invariant family of angles, and
no two of them are separated by the convex hull of a class from $\hlam$. Hence
$\bigcup \Bc$ is a subset of an invariant infinite gap $U$ of $\hlam$.

The gap $U$ cannot be Siegel or caterpillar
as then it will have an isolated in $\hlam$ critical edge which cannot belong to the edges of rational gap-leaves of $\lam_P$.
Hence $U$ is a quadratic invariant non-caterpillar gap which is tuned by leaves of $\lam_P$. By Lemma \ref{l:fxptgm}
there must exist an invariant gap-leaf of $\lam_P$ inside $U$ or an infinite invariant gap $V\subset U$ of $\lam_P$;
since the former is impossible by construction, then $\lam_P$ has an infinite invariant gap $V$. This
completes the proof of the first claim of the lemma.

Now let $f$ belong to the T-class of $P$. We claim that $f$ belongs to the central fiber
of polynomials. Assume that $\lam_f$ and $\lam_P$ are nonempty.
By definition and by Theorem \ref{t:no-share-1} $\lam_f$ belongs to the same alliance
(equivalently, the same $T$-class of laminations)
as $\lam_P$; by Theorem \ref{t:no-share-1} and by the first claim of the lemma,
the %T-class
alliance of laminations of $\lam_P$ is contained in the central alliance of laminations.
Hence $\lam_f$ is central.
By Lemma \ref{l:central2}, this implies that $f$ belongs to the central fiber of polynomials.
\end{proof}

To deal with the converse claim, let $\approx$ be a laminational
equivalence relation such that $\lam_\approx$ has no Siegel gaps but has a flower-like set.
By Theorem \ref{t:kiwi1}, there are polynomials $f$ with $\approx=\sim_f$
($\sim_f$ is an equivalence relation among the arguments
of external rays defined by the non-disjointness of impressions of these rays).
We claim that if $\lam_\approx$ is as above and minimal, then $f$ can be chosen to have a
neutral fixed point. Let $F_\approx$ be the flower-like set of $\lam_\approx$.
Its \emph{combinatorial rotation number} is $0$ unless $F_\approx$ consists of a gap-leaf $G$
and a cycle (or two cycles) of more than one quadratic Fatou gaps attached to $G$.
In the latter case, the combinatorial rotation number of $F_\approx$ is the same as that of $G$.

\begin{thm}
  \label{t:approx-an}
Let $\approx$ be a laminational equivalence relation such that $\lam_\approx$ is minimal,
has no Siegel gaps, but \emph{has} a flower-like set $F_\approx$.
Set $\mu=e^{2\pi i\rho}$, where $\rho$ is the combinatorial rotation number of $F_\approx$.
There exists a cubic polynomial $P(z)=\mu z+b z^2+z^3$ with $\sim_P=\approx$ and parabolic fixed point $0$.
\end{thm}

The following lemma is a special case of Theorem \ref{t:approx-an}, assuming that all
 critical sets of $\approx$ have finite forward orbits.

\begin{lem}
\label{l:approx-an}
Let $\approx$ and $\mu$ be as in Theorem \ref{t:approx-an} except that $\lam_\approx$ does not have to be minimal,
 and assume that all critical gaps or leaves of $\lam_\approx$ are periodic or preperiodic.
Then there is a geometrically finite $f_*(z)=\mu z+b z^2+z^3$ with $\sim_{f_*}=\approx$.
\end{lem}

A polynomial $P$ is \emph{geometrically finite} if all critical points of $P$ in $J_P$
 have finite forward orbits.

\begin{proof}
We will use the \emph{pinching deformation} of Guizhen Cui and Tan Lei \cite{Cui2018}
 (instead of \cite{Cui2018}, one can use a result of Haissinsky \cite{HAI98} that specifically deals with polynomials).
Let $f$ be a cubic polynomial such that $\approx=\sim_f$ (such $f$ exists by Theorem \ref{t:kiwi1}).
By a suitable affine conjugacy, arrange that $0$ is the fixed point of $f$ corresponding to $G$
 (in the sense that the external rays whose arguments represent the vertices of $G$ land on $0$).
If this point is parabolic for $f$, then set $f_*=f$, and we are done.
Otherwise, the fixed point $0$ of $f$ is repelling, and we will make it into a parabolic point via a pinching surgery.

By \cite[Theorem 1.3]{Cui2018}, there exists a continuous one-parameter family $\{\phi_t\}$, $t\in [0,1)$
 of quasiconformal maps $\phi_t:\C\to\C$ with the following properties:
\begin{itemize}
  \item for each $t\in [0,1)$, the map $\phi_t$ conjugates $f$ with a cubic monic polynomial
   $f_t=\phi_t\circ f\circ \phi_t^{-1}$;
  \item on $\C\sm K_{f_t}$, the map $\phi_t^{-1}$ is holomorphic; $\phi_t(0)=0$;
  \item as $t\to 1$, the polynomial $f_t$ converges uniformly to a cubic polynomial $f_*$ such that 0
   is a parabolic fixed point of $f_*$;
  \item the maps $\phi_t$ converge uniformly to a continuous map $\phi$;
  \item restricting $\phi$ to $J_f$, we obtain a topological conjugacy between $f|_{J_f}$ and ${f_*}|_{J_{f_*}}$.
\end{itemize}
Let $\psi_t:\C\sm K_{f_t}\to \C\sm\ol\D$ be the B\"ottcher map, and similarly with $\psi_*:\C\sm K_{f_*}\to\C\sm\ol\D$.
Note that $\phi_t=\psi^{-1}_t\circ\psi_0$ and $\phi=\psi_*^{-1}\circ\psi_0$ on $\C\sm K_{f}$.
It follows easily that $\phi:\C\sm K_f\to\C\sm K_{f_*}$ is a homeomorphism that conjugates $f$ with $f_*$.
Since by the above the restriction of $\phi$ on $J_f$ topologically conjugates $f|_{J_f}$ and ${f_*}|_{J_{f_*}}$ then
external rays $R_{f}(\al)$ and $R_f(\be)$ land on a common point if and only if
$R_{f_*}(\al)$ and $R_{f_*}(\be)$ land on a common point.

The maps $f$ and $f_*$ are topologically conjugate on their Julia sets; it follows that the
 parabolic fixed point $0$ has the same combinatorial rotation number as the corresponding fixed point $a$ of $f$.
We see that $f'_*(0)=\mu$.
A linear change of variables then takes $f_*$ to the form $f_*(z)=\mu z+bz^2+z^3$ for some $b\in\C$.
\end{proof}

Recall that a cubic polynomial $P$ with \emph{degenerate parabolic} fixed point $0$
has two cycles of petals at $0$.

\begin{lem}
\label{l:PZ}
Consider cubic polynomials $P, \{P_n\}$ from $\Cc_3$.
Suppose that one of the following possibilities takes place: either $P$ has no parabolic fixed points, or $P$ has two distinct parabolic fixed points,
or $P$ has one degenerate parabolic point. If $\al$ is $\si_3$-periodic, $P_n\to P$, $R_{P_n}(\al)$ lands on a point
$z_n$, and $z_n$ converge to a $P$-fixed point $z$, then $R_P(\al)$ lands on $z$.
\end{lem}

\begin{proof}
By Lemma \ref{l:rep} we may assume that $P$ has parabolic points. Moreover, without loss of generality we may assume
that $P$ has parabolic fixed points $0$ and $z$ so that either $z\ne 0$, or $z=0$ is a degenerate parabolic fixed point.
Set $R_n(\al)=R_{P_n}(\al)$.
Passing to a subsequence if necessary, assume that there is a limit of $R_n(\al)\cup\{z\}$ in the Hausdorff metric
associated with the spherical metric on $\C\cup\{\infty\}$.
Denote this limit by $R^{\lim}(\al)$.
Clearly, $R_P(\al)\subset R^{\lim}(\al)$, but the difference $R^{\lim}(\al)\sm R_P(\al)$ may a priori be
larger than just the landing point of $R_P(\al)$.
Theorem A of Petersen--Zakeri \cite{PZ} gives a description of such limits $R^{\lim}(\al)$.
As follows from this theorem and the Basic Structure Lemma \cite{PZ}
(note that different critical points of $P$ are in different parabolic Fatou domains),
the limit set $R^{\lim}(\al)$ is the union of $R_P(\al)$ and an at most countable set of loops based on the landing point of $R_P(\al)$.
Since $z\in R^{\lim}(\al)$, then $R_P(\al)$ lands on $z$ as claimed.
\end{proof}

Let us prove Theorem \ref{t:approx-an} under an additional assumption of ``rationality''.

\begin{lem}
  \label{l:appr-rat1}
The conclusion of Theorem \ref{t:approx-an} holds under the additional assumption that $F_\approx$ is not
a quadratic invariant gap of regular critical type.
\end{lem}

Observe that $F_\approx$ cannot be a caterpillar gap because it is a gap of $\lam_\approx$.

\begin{proof}
If $\lam_\approx$ is countable, then by Lemma \ref{l:count0} and by the assumptions $\lam_\approx$ satisfies all the conditions
of Lemma \ref{l:approx-an}, which implies the desired.

If $\lam_\approx$ is perfect, then, by Lemma \ref{l:no-perf-cent}, the set
 $F_\approx=V$ is a quadratic invariant gap with major $M=\ol{uv}$; by
 the assumptions, $M$ is periodic.
Set $\lam:=\lam_\approx$. As $\lam$ is minimal, by Lemma \ref{l:dense}
it is the closure of the set of all iterated pullbacks of $M$. Let $C\ne V$ be the other critical set of $\lam$.
If $C$ is periodic or preperiodic, then the desired holds by Lemma \ref{l:approx-an}. If $C$ is a Fatou gap then $C$ is periodic or preperiodic. Hence we may assume
that $C$ is a critical gap-leaf of $\lam$ with infinite orbit; since $\lam$ is minimal,
edges of $C$ can be approximated by iterated pullbacks of $M$. Thus, $\lam$ has only one cycle %orbit
of Fatou gaps, namely $\{V\}$, and all infinite gaps of $\lam$ are iterated pullbacks of $V$.

Form the sequence of laminations $\lam_n$ as follows.
First, take all iterated pullbacks $M$ in $\lam$ of level at most $n$.
Let $\lam'_n$ be the thus obtained forward invariant lamination.
Next, take a critical chord $c_n$ compatible with $\lam'_n$ and mapping eventually to an endpoint of $M$.
Set $\lam''_n=\lam'_n\cup\{c_n\}$; this is a forward invariant cubic lamination.
There is only one way of taking iterated pullbacks of leaves of $\lam''_n$ so that they do not cross
the interior of $V$ or the chord $\{c_n\}$.
Adding all these iterated pullbacks and their limits to $\lam''_n$ yields a cubic invariant lamination $\lam'''_n$.

By construction $\lam'''_n$ is proper and $V$ is a gap of $\lam'''_n$.
Define the
%laminational equivalence
 relation $\approx_n$ as $\approx_{\lam'''_n}$ (see Definition \ref{d:fineq});
 by Theorem \ref{t:proper}; it %$\approx_n$
 is a laminational equivalence relation. Set $\lam_n=\lam_{\approx_n}$.
All critical sets of $\lam_n$ are periodic or preperiodic: one critical set is $V$, and the other one
is the finite gap-leaf $C_n\supset c_n$ (since $\approx_n$ is a laminational equivalence relation, $\approx_n$-classes
are finite, and $C_n$ is the convex hull of such a class).
Using Lemma \ref{l:approx-an}, choose a polynomial $P_n(z)=\mu z+b_nz^2+z^3$ with $\la(P_n)=\approx_n$ (in our case $\mu=1$).
Note that every $P_n$ has a critical point $\om_n$ (associated to $C_n$) that eventually maps to a periodic point of $P_n$.

Passing to a subsequence if necessary, assume that $P_n$ converge to a cubic polynomial $P=\mu z+bz^2+z^3$. Consider
cases; set $\lam_P:=\lam_{\sim_P}$

(1) Let $0$ be a nondegenerate parabolic point of $P$. We claim that $M$ is in $\lam_P$, i.e.
that $R_P(u)$ and $R_P(v)$ land on the same point. Let $R_{P_n}(u)$ and
$R_{P_n}(v)$ land on a point $x_n$ (since $M$ is a leaf of each $\lam_n$ then these rays land on the same point);
let $R_P(u)$ land on a point $y$ while $R_P(v)$ land on a point $z$. We need to show that $y=z$.

Lemma \ref{l:PZ} covers the cases when $P$ has two distinct parabolic fixed points or one degenerate parabolic fixed point.
%By Lemma \ref{l:PZ} it suffices to consider the case when
It remains to assume that $P$ has only one parabolic fixed point, and it is nondegenerate.
Then the only parabolic fixed point of $P$ is $0$; assume that $z=0$ while $y\ne 0$ is a repelling periodic point.
However if $v$ is not $\si_3$-fixed this is impossible (a non-fixed periodic ray cannot land on a fixed point with multiplier $1$).
Thus, we may assume that $v=0$, hence $u=1/2$, and, say, $V$ is the quadratic gap with major $M=\hdi$ and located above $M$.
However, by Lemma \ref{l:approx-an}, the leaf $M=\hdi$ of $\lam_n$ then corresponds to the parabolic point $0$
 of each $P_n$;  a contradiction with Lemma \ref{l:PZ} and the assumptions about the point $y$.

By construction and definition of $\lam_P$ it follows that for every $n$ all level $\le n$ pullbacks of $M$ in $\lam_P$
coincide with those of $\lam_n$, hence also with those of $\lam$. We conclude that $\lam\subset\lam_P$.
If $\lam_P$ is minimal, then we are done as $\lam=\lam_P$ in this case.
Suppose $\lam_P$ is not minimal, that is, it tunes the lamination $\lam$ in a nontrivial way, i.e. leaves
of $\lam_P$ are diagonals of $V$. We claim that this is impossible. Indeed, the map on $V$ is two-to-one.
Since the fixed point $0$ is parabolic and has to have a parabolic domain associated to it, this leaves no room for any tuning of
$V$, a contradiction.

(2) Suppose that the parabolic point $0$ of $P$ is degenerate,
that is, $P$ has two cycles of parabolic petals. This implies that
$\lam_{\sim P}$ has a leaf $\hdi$ and two quadratic invariant gaps sharing the same major:
the gap $V$ above $\hdi$ and the gap $U$ below $\hdi$. Let $\{U,\hdi, V\}=F_P$; the presence of sets
from $F_P$ in a lamination completely defines the lamination.
%By Lemma \ref{l:PZ}, $G=\hdi$.
If $\lam\ne \lam_P$ then
there must exist an iterated pullback $\ell$ of $M$
 which is a leaf of $\lam_n$ for all large $n$ located close to $\hdi$ and, say, below $\hdi$
(otherwise $\lam=\lam_P$).

By \cite[Theorem 7.5.2]{bfmot13},
there is an invariant leaf, finite gap, or infinite gap $T$ of $\lam$ located below $M$ and disjoint from $M$.
The set $T$ cannot be an infinite gap as otherwise it is easy to see that $T$ coincides with $U$ \cite{bopt16},
 a contradiction with the existence of $\ell$.
Hence $T$ is a gap-leaf. Since pullbacks of $M$ separate $C$ and $T$, then $T$ is a gap-leaf of $\lam_n$ for large $n$.
By Lemma \ref{l:PZ}, rays whose arguments are vertices of $T$ land on the same fixed repelling point of $P$ which contradicts
the existence of $U$ in $\lam_P$.
\end{proof}

\begin{comment}

In the dynamical plane of $P_n$, the external rays with arguments $\al$, $\be$ land at a point $z_n$
that is eventually mapped to $0$, say $P_n^k(z_n)=0$.
Note that $k$ does not depend on $n$.
The external rays of $P$ with arguments $\al$, $\be$ both land close to $z_n$
(by the Inverse Function Theorem and the fact that no critical point of $P$ is an eventual preimage of $0$).
Since the set $P^{-k}(0)$ is finite, it follows that $R_P(\al)$ and $R_P(\be)$ land at the same point.
It follows that $\ell$ is a compatible with $\lam_{\sim P}$, a contradiction.

\end{comment}

We can now complete the proof of Theorem \ref{t:approx-an}.

\begin{proof}[Proof of Theorem \ref{t:approx-an}]
It remains only to consider the case when $\lam=\lam_\approx$ is the canonical lamination of
 a quadratic invariant gap of regular critical type (that is, the only cubic lamination containing this gap).
Consider the space $\Fc_1$ of all cubic polynomials of the form $f_{1,b}(z)=z+bz^2+z^3$.
Let $\mathcal{H}_1$ be the interior component of the connectedness locus in $\Fc_1$ consisting of
 polynomials with both critical points in the same parabolic Fatou domain attached to $0$.
By \cite[Theorem A]{zha22}, the boundary of $\mathcal{H}_1$ is a Jordan curve.
For every $f\in\bd(\mathcal{H}_1)$, one critical point $\om_1(f)$ lies in the immediate basin $B_f(0)$ of $0$,
 and the other critical point $\om_2(f)$ lies either in the boundary of $B_f(0)$ or
 in a parabolic basin attached to a boundary point of $B_f(0)$.
This is a consequence of the following claim proved in \cite{Roe10} (Theorem 1 and Section 6):
 the boundary $\bd(B_f(0))$ is a Jordan curve, on which $f$ is canonically topologically conjugate with
 the angle doubling map of the circle.
Write $\nu_2(f)$ for the image of $\om_2(f)$ (or of the parabolic point whose immediate basin contains $\om_2(f)$)
 under this conjugacy.
As $f$ loops around $\bd(\mathcal{H}_1)$, the point $\nu_2(f)$ moves continuously on the circle.
In fact, $\nu_2(f)$ makes at least one full loop.
This follows from the Argument Principle and the fact that the function $f\mapsto \om_1(f)-\om_2(f)$
 has at least one zero and no poles in $\mathcal{H}_1$.
We may therefore choose $f\in\bd(\mathcal{H}_1)$ so that to make $\nu_2(f)$ equal any
 prescribed point of the circle.

Now let $M$ be the major of the invariant quadratic gap $U$ of $\lam$.
By our assumption, $M$ is a non-periodic critical leaf.
There is a monotone projection $\psi_U:\uc\to\uc$ that collapses all components of $\uc\sm U$
 and that semi-conjugates $\si_3$ with $\si_2$ on $U\cap\uc$.
In particular, $\psi_U(M)$ is a well-defined point of $\uc$.
Choose $f_*\in\bd(\mathcal{H}_1)$ so that $\psi_U(M)=\{\nu_2(f_*)\}$.
It is easy to see now that the full lamination of $f_*$ coincides with $\lam$.
\end{proof}

Corollary \ref{c:neutrep} relates various types of T-classes of polynomials.

\begin{cor}\label{c:neutrep}
The T-class of any polynomial with an attracting fixed point contains a polynomial with a neutral fixed point.
%The union of the T-classes of polynomials with a non-repel\-ling fixed point coincides with the union of T-classes
%of polynomials with a neutral fixed point.
\end{cor}

\begin{proof}
It suffices to consider the T-class of a polynomial $f$ with a (super)attracting fixed point.
Also, we may assume that the T-class of $f$ is nontrivial; the trivial
class contains the polynomial $f_1(z)=z+\sqrt 3 z+z^3$, by \cite[Theorem 1]{Roe10}.
By Theorem \ref{t:bco1}, the corresponding lamination $\lam_f$ is nonempty and has a quadratic invariant gap $U$.
Consider a minimal sublamination $\hlam$ of $\lam_f$. By definition,
$\hlam$ is nonempty and has a gap $V\supset U$. It follows that $V=U$. Moreover, $\hlam$ cannot contain
a Siegel gap $W$ as otherwise we can remove all edges of $W$ with all their iterated pullbacks and still have a nonempty sublamination of $\lam_f$
(because $U$ will remain after that removal).
By Theorem \ref{t:approx-an}, there exists a polynomial $g$ with a parabolic
fixed point such that $\lam_g=\hlam$, and $f$ belongs to the T-class of $g$ as desired.
\end{proof}

%Let us now prove Theorem \ref{t:central-f}.

\begin{proof}[Proof of Theorem \ref{t:central-f}]
Let $f$ belong to the central fiber of polynomials. By Corollary \ref{c:neutrep} and Lemma \ref{l:rep-triv} %if $\lam^{rep}_f$ is empty
%or $\lam_f$ is empty then $f$ belongs to the union of T-classes of polynomials with a neutral fixed point.
%Suppose
we may assume that $\lam^{rep}_f$ is nonempty and $\lam_f$ is nonempty. By Lemma \ref{l:central2} $\lam_f$ is central.
Hence, by Lemma \ref{l:rep-nontriv}, $f$ belongs to the T-class of a polynomial $P$, and
there are two cases.

(1) If $P$ has a non-repelling fixed point then by Corollary \ref{c:neutrep}
$P$ and $f$ belong to the T-class of a polynomial with a neutral fixed point.

(2) $\lam_P$ is minimal, has no Siegel or caterpillar gaps, but has a flower-like set.
Then by Theorem \ref{t:approx-an}, there exists a polynomial $g$ with a parabolic
fixed point such that $\lam_g=\hlam$.

Thus, the central fiber of polynomials is contained in the union of T-classes of polynomials with neutral fixed point.
The opposite inclusion holds by Lemma \ref{l:1way1}.
\end{proof}

We can now prove the Main Theorem.

\begin{proof}[Proof of the Main Theorem]
The map $\eta_3$ is introduced in Definition \ref{d:pi} (fi\-bers of polynomials are $\eta_3$-point preimages).
This map associates to a po\-ly\-nomial $f\in \Cc_3$ its fiber of critical portraits (introduced in Definition \ref{d:poly-fiber}).
By Theorem \ref{t:main}, the map $\eta_3$ is continuous. By Lemma \ref{l:regfibren}, all non-central fibers of polynomials are connected.
By Corollary \ref{c:non-central}, all non-central fibers of polynomials are T-classes. By Theorem \ref{t:central-f}, the central fiber
of polynomials is the union of T-classes of polynomials with a non-repelling fixed point equal (by Corollary \ref{c:neutrep}) to the union
of T-classes of polynomials with a neutral fixed point. It follows that $\eta_3$ is T-stable.
Since by Definition \ref{d:mcd}, any T-stable map $\Psi:\Cc_3\to Y$ collapses any T-class of polynomials to a point, and the union
of the T-classes of the family of all polynomials with non-repelling fixed points to one point, then
%$\Psi:\Cc_3\to Y$ can be represented as
$\Psi=\widehat \Psi\circ \eta_3$ as desired.
%, and
%one can think of $\Psi$ as the result of first applying $\eta_3$ and mapping $\Cc_3$ to $X_3$, and then,
%possibly, identifying some of the points of $X_3$  by applying $\widetilde \Psi$.
\end{proof}

\end{document}